\documentclass[12pt]{article}
\usepackage{amsmath}
\usepackage{amssymb}
\usepackage{amsthm}
\usepackage{amscd, newlfont, enumerate}
\usepackage[latin2]{inputenc}
\usepackage[dvips]{graphicx}
\usepackage[cmtip,arrow]{xy}
\usepackage[active]{srcltx}
\usepackage{color}

\usepackage{hyperref}

 \newtheorem{thrm}{Theorem}[section]
 
 \newtheorem{lemma}[thrm]{Lemma}
 \newtheorem{proposition}[thrm]{Proposition}
 \newtheorem{definition}[thrm]{Definition}
 \newtheorem{remark}[thrm]{Remark}

\def\rt{\rtimes}

\def\vh{\frac{1}{2}}

\newcommand{\Irr}{\mathop{\text{Irr}}}

\title{Degenerate principal series in the general case}
\author{Yeansu Kim, Baiying Liu, and Ivan Mati\'{c}}

\date{\today}

\begin{document}

\maketitle

\begin{abstract}
Let $G_n$ denote either the group $SO(2n+1, F)$, $Sp(2n, F)$, or $GSpin(2n+1, F)$ over a non-archimedean local field of characteristic different than two. We determine all composition factors of degenerate principal series of $G_n$, using methods based on the Aubert involution and known results on irreducible subquotiens of the generalized principal series of particular type.
\end{abstract}

{\renewcommand{\thefootnote}{} \footnotetext[1]{\textit{MSC2000:}
primary 22E35; secondary 22E50, 11F70}
\footnotetext[2]{\textit{Keywords:} classical $p$-adic groups, degenerate principal series, generalized principal series}}

\section{Introduction}

Let $F$ be a non-archimedean local field of characteristic different than two.
Let $\mathrm{G}_n$ denote a symplectic, odd special orthogonal, or odd general spin group of split rank $n$ defined over $F$, and $G_n=\mathrm{G}_n(F)$. 
The aim of this paper is to obtain a uniform description of reducibility and composition factors of degenerate principle series of $G_n$. This greatly generalizes and simplifies previous works of Jantzen \cite{Jan7}, Kudla-Rallis \cite{KR1}, Gustafson \cite{Gust}, and others. We note that the degenerate principle series, besides being interesting by themselves, play an important role in the theory of automorphic forms, especially
the extension of the Siegel-Weil formula, constructions of residual spectrum \cite{JK01, Kim01} , and in the local theta-correspondence.  

Let $\sigma$ denote an irreducible cuspidal representation of some $G_n$. Also, let $\rho_0$ denote an irreducible cuspidal representation of $GL(n_{\rho_0}, F)$, and let $\rho$ denote an irreducible
self-contragredient, essentially self-contragredient (i.e., $\widetilde\rho\cong\rho\otimes\omega_{\sigma}$), cuspidal representation of $GL(n_{\rho}, F)$ when $G_n$ is a classical group, $GSpin(2n+1,F)$, respectively. Then there exist unique non-negative half-integers $\alpha, \beta$ such that $\nu^{\alpha} \rho \rtimes \sigma$, $\nu^{\beta} \rho_0 \rtimes \sigma$ are reducible (for more details regarding the notation we refer the reader to Section 2). For $x \geq \alpha$ such that $x - \alpha \in \mathbb{Z}$, the induced representation
$\nu^{-x} \rho \times \nu^{-x+1} \rho \times \cdots \times \nu^{-\alpha} \rho \rtimes \sigma$ contains a unique irreducible subrepresentation, which we denote by $\zeta(\rho, x; \sigma)$. A degenerate principal series is an induced representation of the form
\begin{equation} \label{dps}
\zeta([\nu^{-b} \rho_0, \nu^{-a} \rho_0]) \rtimes \zeta(\rho, x; \sigma),
\end{equation}
for $a, b$ such that $b - a \in \mathbb{Z}$, where $\zeta([\nu^{-b} \rho_0, \nu^{-a} \rho_0])$ is a Zelevinsky segment representation, i.e., the unique irreducible subrepresentation of $\nu^{-b} \rho_0 \times \nu^{-b+1} \rho_0 \times \cdots \times \nu^{-a} \rho_0$. It has been explained in detail in \cite[Section~2]{Jan7} that this definition generalizes the classical notion of the degenerate principal series, studied in \cite{Gust} and \cite{KR1}. We note that the composition series of the degenerate principal series (\ref{dps}) have been determined in \cite{Jan7} for $\alpha \in \{ 0, \vh, 1 \}$, using Tadi\'{c}'s Jacquet modules method \cite{Tad5, Tad2}, and here we treat the general case. Our results show that the degenerate principal series are multiplicity one representations of length up to four, and also provide a deeper insight into the structure of the irreducible subquotients.

Our approach to the determination of reducibility and composition factors of induced representations of the form (\ref{dps}) is completely different than one used in \cite{Jan7}, and is based on the methods of the Aubert involution. The Aubert dual of the degenerate principal series is a special type of the generalized principal series, and the composition factors of such representations have been determined in \cite{Mu3} and \cite[Proposition~3.2]{Matic8}. To determine the Aubert duals of composition factors in question, we use a further adjustment of the methods initiated in \cite{Matic14, Matic15, Matic10}. Eventually, it turns out that needed Aubert duals of tempered representations mostly follow directly from \cite{Matic14, Matic10}. On the other hand, to determine the Aubert duals of the involved non-tempered representations we use an inductive approach based on the detailed investigation of embeddings and Jacquet modules of such representations, using a case-by-case consideration.

Let us now describe the contents of the paper in more detail. In the following section we present some preliminaries, while the first special case $\beta=0$ is treated in the third section. The case $\beta>0$ is studied in Sections 4 -- 6, where in the fourth section we handle the case $a \geq 1$, in the fifth section the case $a \leq 0$, and in the sixth section we deal with the case $a = \vh$. To work effectively, from Lemma 2.5 to the end of Section 6, we mainly focus on the cases $G_n=Sp(2n,F)$ and $SO(2n+1,F)$ (see Remark 2.4). 
In the final section we provide necessary adjustments in the odd GSpin case.

\subsection*{Acknowledgements}
The first author is supported by Chonnam National University (Grant number: 2018-0978).
The second author is partially supported by NSF grants DMS-1702218, DMS-1848058, and by start-up funds from the Department of Mathematics at Purdue University.
The third author is partially supported by Croatian Science Foundation under the project IP-2018-01-3628.

\section{Preliminaries}

Throughout the paper, $F$ will denote a non-archimedean local field of characteristic different than two.

For a connected reductive $p$-adic group $G$ defined over field $F$, let $\Sigma$ denote the set of roots of $G$ with respect to fixed minimal parabolic subgroup and let $\Delta$ stand for a basis of $\Sigma$. For $\theta \subseteq \Delta$, we let $P_{\theta}$ denote the standard parabolic subgroup of $G$ corresponding to $\theta$ and let $M_{\theta}$ denote a corresponding standard Levi subgroup. Let $W$ denote the Weyl group of $G$.

For a parabolic subgroup $P$ of $G$ with the Levi subgroup $M$, and a representation $\sigma$ of $M$, we denote by $i_{M}(\sigma)$ a normalized parabolically induced representation of $G$ induced from $\sigma$. Also, let $r_M(\sigma)$ stand for the normalized Jacquet module of an admissible finite length representation $\sigma$ of $G$, with respect to the standard parabolic subgroup having the Levi subgroup equal to $M$.

We take a moment to recall the definition of the Aubert involution and some of its basic properties \cite{A, A2}.

\begin{thrm} \label{aub}
Define the operator on the Grothendieck group of admissible representations of finite length of $G$ by
\begin{equation*}
D_{G} = \sum_{\theta \subseteq\Delta} (-1)^{| \theta |} i_{M_{\theta}} \circ r_{M_{\theta}}.
\end{equation*}
Operator $D_{G}$ has the following properties:
\begin{enumerate}[(i)]
  \item $D_{G}$ is an involution.
  \item $D_{G}$ takes irreducible representations to irreducible ones.
  \item If $\sigma$ is an irreducible cuspidal representation, then $D_{G}(\sigma) = (-1)^{|\Delta|} \sigma$.
  \item For a standard Levi subgroup $M = M_{\theta}$, we have
\begin{equation*}
r_{M} \circ D_{G} = Ad(w) \circ D_{w^{-1}(M)} \circ r_{w^{-1} (M)},
\end{equation*}
where $w$ is the longest element of the set $\{ w \in W : w^{-1}(\theta) > 0 \}$.
  \item For a standard Levi subgroup $M = M_{\theta}$, we have $D_G \circ i_{M} = i_{M} \circ D_M$.
\end{enumerate}
\end{thrm}

We look at the usual towers of symplectic or orthogonal groups $G_n = G(V_n)$, that are groups of isometries of $F$-spaces $(V_n, ( ~, ~)), n \geq 0$, where the form $( ~, ~)$ is non-degenerate and it is skew-symmetric if the tower is symplectic and symmetric otherwise. In the final section, we also consider the odd general spin groups $G_n=GSpin(2n+1, F)$ (See Section \ref{GSpin} for the definition). The set of standard parabolic subgroups of the group $G_n$ will be fixed in the usual way.

Then the Levi subgroups of standard parabolic subgroups have the form $M \cong GL(n_{1}, F) \times \cdots \times GL(n_{k}, F) \times G_{m}$, where $GL(n_i, F)$ denotes a general linear group of rank $n_i$ over $F$. For simplicity of exposition, if $\delta_{i}, i = 1, 2, \ldots, k$ denotes a representation of $GL(n_{i}, F)$, and if $\tau$ stands for a representation of $G_{m}$, we let $\delta_{1} \times \delta_2 \times \cdots \times \delta_{k} \rtimes \tau$ stand for the induced representation $i_{M}(\delta_{1} \otimes \delta_2 \otimes \cdots \otimes \delta_{k} \otimes \tau)$ of $G_n$, where $M$ is the standard Levi subgroup isomorphic to $GL(n_{1}, F) \times \cdots \times GL(n_{k}, F) \times G_{m}$. Here $n = n_{1} + n_{2} + \cdots + n_{k} + m$.

Similarly, by $\delta_{1} \times \delta_2 \times \cdots \times \delta_{k}$ we denote the induced representation $i_{M'}(\delta_{1} \otimes \delta_2 \otimes \cdots \otimes \delta_{k})$ of the group $GL(n', F)$, where the Levi subgroup $M'$ equals $GL(n_{1}, F) \times GL(n_{2}, F) \times \cdots \times GL(n_{k}, F)$ and $n' = n_{1} + n_{2} + \cdots + n_{k}$.

Let $\Irr(GL(n,F))$ denote the set of all irreducible admissible representations of $GL(n,F)$, and let $\Irr(G_n)$ denote the set of all irreducible admissible representations of $G_n$. Let $R(GL(n, F))$ stand for the Grothendieck group of admissible representations of finite length of $GL(n, F)$ and define $R(GL) = \oplus_{n \geq 0} R(GL(n, F))$. Similarly, let $R(G_n)$ stand for the Grothendieck group of admissible representations of finite length of $G_{n}$ and define $R(G) = \oplus_{n \geq 0} R(G_n)$.

If $\sigma$ is an irreducible representation of $G_n$, we denote by $\hat{\sigma}$ the representation $\pm D_{G_n}(\sigma)$, taking the sign $+$ or $-$ such that $\hat{\sigma}$ is a positive element in $R(G_{n})$. We call $\hat{\sigma}$ the Aubert dual of $\sigma$.

Using Jacquet modules for the maximal standard parabolic subgroups of $GL(n,F)$, one can define $m^{\ast}(\pi) = \sum_{k=0}^{n} (r_{(k)} (\pi)) \in R(GL) \otimes R(GL)$, for an irreducible representation $\pi$ of $GL(n,F)$, and then extend $m^{\ast}$ linearly to $R(GL)$. Here $r_{(k)}(\pi)$ denotes the normalized Jacquet module of $\pi$ with respect to the standard parabolic subgroup having the Levi subgroup equal to $GL(k, F) \times GL(n-k, F)$, and we identify $r_{(k)}(\pi)$ with its semisimplification in $R(GL(k,F)) \otimes R(GL(n-k, F))$.

Let $\nu$ denote the composition of the determinant mapping with the normalized absolute value on $F$. Let $\rho \in \Irr(GL(k,F))$ denote a cuspidal
representation. By a segment of cuspidal representations we mean a set of the form $\{ \rho, \nu \rho, \ldots, \nu^{m} \rho \}$, which we denote by $[\rho, \nu^{m} \rho]$.

By the results of \cite{Zel}, each irreducible essentially square-integrable representation $\delta \in \Irr(GL(n, F))$ is attached to a segment, and we set $\delta = \delta ([\nu^{a} \rho, \nu^{b} \rho ])$, which is the unique irreducible subrepresentation of $\nu^{b} \rho \times \nu^{b-1} \rho \times \cdots \times \nu^{a} \rho$, where $a, b \in \mathbb{R}$ are such that $b - a$ is a non-negative integer and $\rho$ is an irreducible unitary cuspidal representation of some $GL(k, F)$. The induced representation $\nu^{b} \rho \times \nu^{b-1} \rho \times \cdots \times \nu^{a} \rho$ also contains a unique irreducible quotient, which we denote by $\zeta([\nu^{a} \rho, \nu^{b} \rho ])$. Furthermore, $\zeta([\nu^{a} \rho, \nu^{b} \rho ])$ is the unique irreducible subrepresentation of $\nu^{a} \rho \times \nu^{a+1} \rho \times \cdots \times \nu^{b} \rho$, and in $R(GL)$ we have
\begin{equation*}
\nu^{a} \rho \times \nu^{a+1} \rho = \delta([\nu^{a} \rho, \nu^{a+1} \rho ]) + \zeta([\nu^{a} \rho, \nu^{a+1} \rho ]) 
\end{equation*}
and
\begin{equation*}
\nu^{a} \rho \times \nu^{a+1} \rho \times \nu^{a+1} \rho = \delta([\nu^{a} \rho, \nu^{a+1} \rho ]) \times \nu^{a+1} \rho + \zeta([\nu^{a} \rho, \nu^{a+1} \rho ]) \times \nu^{a+1} \rho,
\end{equation*}
both representations $\delta([\nu^{a} \rho, \nu^{a+1} \rho ]) \times \nu^{a+1} \rho$ and $\zeta([\nu^{a} \rho, \nu^{a+1} \rho ]) \times \nu^{a+1} \rho$ being irreducible.



Let us briefly recall the Langlands classification for classical groups. We favor the subrepresentation version of this classification over the quotient one since it is more appropriate for our Jacquet module considerations. 

For every irreducible essentially square-integrable representation $\delta \in R(GL)$, there is a unique $e(\delta) \in \mathbb{R}$ such that $\nu^{-e(\delta)} \delta$ is unitarizable. Note that $e(\delta([\nu^{a} \rho, \nu^{b} \rho ])) = (a + b) / 2$. Every non-tempered irreducible representation $\pi$ of $G_{n}$ can be written as the unique irreducible (Langlands) subrepresentation of an induced representation of the form $\delta_{1} \times \delta_{2} \times \cdots \times \delta_{k} \rtimes \tau$, where $\tau$ is a tempered representation of some $G_{t}$, and $\delta_1, \delta_2, \ldots, \delta_k \in R(GL)$ are irreducible essentially square-integrable representations such that $e(\delta_1) \leq e(\delta_2) \leq \cdots \leq e(\delta_k) < 0$. In this case, we write $\pi = L(\delta_{1}, \delta_{2}, \ldots, \delta_{k}; \tau)$. For a given $\pi$, the representations $\delta_{1}, \delta_{2}, \ldots, \delta_{k}$ are unique up to a permutation among those $\delta_i$ having the same exponents.

Let $\tau \in R(G)$ denote an irreducible tempered representation. If $\delta_1, \delta_2$, $\ldots, \delta_k \in R(GL)$ are irreducible essentially square-integrable representations such that $e(\delta_i) < 0$ for $i = 1, 2, \ldots, k$, and $\delta_i \times \delta_j \cong \delta_j \times \delta_i$ for $i < j$ such that $e(\delta_i) > e(\delta_j)$, then the induced representation $\delta_1 \times \delta_2 \times \cdots \times \delta_k \rtimes \tau$ contains a unique irreducible subrepresentation, which will also be denoted by $L(\delta_1, \delta_2, \ldots, \delta_k; \tau)$, for simplicity of the notation.

For a representation $\sigma \in R(G_{n})$ and $1 \leq k \leq n$, we denote by $r_{(k)}(\sigma)$ the normalized Jacquet module of $\sigma$ with respect to the parabolic subgroup $P_{(k)}$ having the Levi subgroup equal to $GL(k, F) \times G_{n-k}$. We identify $r_{(k)}(\sigma)$ with its semisimplification in $R(GL(k,F)) \otimes R(G_{n-k})$ and consider
\begin{equation*}
\mu^{\ast}(\sigma) = 1 \otimes \sigma + \sum_{k=1}^{n} r_{(k)}(\sigma) \in R(GL) \otimes R(G).
\end{equation*}

We pause to state a result, derived in \cite{Tad5} (\cite{Kim1} for odd $GSpin$ groups), which presents a crucial structural formula for our calculations of Jacquet modules of classical groups.

\begin{lemma} \label{osn}
Let $\rho \in \Irr(GL(n,F))$ denote a cuspidal representation and let $k, l \in \mathbb{R}$ such that $k + l$ is a non-negative integer. Let $\sigma \in R(G)$ denote an admissible representation of finite length, and write $\mu^{\ast}(\sigma) = \sum_{\tau, \sigma'} \tau \otimes \sigma'$. Then the following holds:
\begin{align*}
\mu^{\ast}(\delta([\nu^{-k} \rho, \nu^{l} \rho]) \rtimes \sigma) & = \sum_{i = -k-1}^{l} \sum_{j=i}^{l} \sum_{\tau, \sigma'} \delta([ \nu^{-i} \widetilde{\rho}, \nu^{k} \widetilde{\rho} ]) \times  \delta([\nu^{j+1} \rho, \nu^{l} \rho]) \times \tau \nonumber \otimes \\
&  \qquad \otimes  \delta([\nu^{i+1} \rho, \nu^{j} \rho]) \rtimes \sigma'.
\end{align*}
If $\sigma$ is an admissible representation of finite length of the odd GSpin group, we have
\begin{align*}
\mu^{\ast}(\delta([\nu^{-k} \rho, \nu^{l} \rho]) \rtimes \sigma) & = \sum_{i = -k-1}^{l} \sum_{j=i}^{l} \sum_{\tau, \sigma'} \delta([ \nu^{-i} \widetilde{\rho} \otimes ( \omega_{\sigma} \circ det ), \nu^{k} \widetilde{\rho} \otimes ( \omega_{\sigma} \circ det ) ])\\
 & \qquad \times  \delta([\nu^{j+1} \rho, \nu^{l} \rho]) \times \tau \nonumber \otimes  \delta([\nu^{i+1} \rho, \nu^{j} \rho]) \rtimes \sigma',
\end{align*}
where $\omega_{\sigma}$ denotes the central character of $\sigma$.
We omit $\delta([\nu^{x} \rho, \nu^{y} \rho])$ if $x > y$.
\end{lemma}

An irreducible representation $\sigma \in R(G)$ is called
strongly positive if for every embedding
$$ \sigma \hookrightarrow \nu^{s_{1}} \rho_{1} \times \nu^{s_{2}} \rho_{2} \times \cdots \times \nu^{s_{k}} \rho_{k} \rtimes \sigma_{cusp},$$
where $\rho_{i} \in R(GL(n_{\rho_{i}},F))$, $i = 1, 2, \ldots, k$, are
cuspidal unitary representations and $\sigma_{cusp}
\in R(G)$ is an irreducible cuspidal representation, we
have $s_{i} > 0$ for each $i$.

Let us briefly recall an inductive description of
non-cuspidal strongly positive discrete series, which has
been obtained in \cite{Kim1, Matic3, MT1}.

\begin{proposition}  \label{spds}
Suppose that $\sigma_{sp} \in R(G)$ is an
irreducible strongly positive representation and let $\rho \in
R(GL)$ denote an irreducible
cuspidal unitary representation such that some twist of $\rho$ appears in
the cuspidal support of $\sigma_{sp}$. We denote by $\sigma_{cusp}$ the
partial cuspidal support of $\sigma_{sp}$. Then there exist unique $a,
b \in \mathbb{R}$ such that $a > 0, b > 0$, $b - a \in
\mathbb{Z}_{\geq 0}$, and a unique irreducible strongly positive
representation $\sigma'_{sp}$ without $\nu^{a} \rho$ in the cuspidal support, with the property that $\sigma_{sp}$ is the
unique irreducible subrepresentation of $\delta([\nu^{a} \rho,
\nu^{b} \rho]) \rtimes \sigma'_{sp}$. Furthermore, there is a
non-negative integer $l$ such that $a + l = s$, for $s > 0$
such that $\nu^{s} \rho \rtimes \sigma_{cusp}$ reduces. If $l = 0$,
there are no twists of $\rho$ appearing in the cuspidal support of
$\sigma'_{sp}$ and if $l > 0$ there exist unique $b' > b$ and a 
unique strongly positive discrete series $\sigma''_{sp}$, which
contains neither $\nu^{a} \rho$ nor $\nu^{a+1} \rho$ in its
cuspidal support, such that $\sigma'_{sp}$ can be written as the unique
irreducible subrepresentation of $\delta([\nu^{a+1} \rho, \nu^{b'}
\rho]) \rtimes \sigma''_{sp}$.
\end{proposition}

Through the paper, we fix an irreducible cuspidal representation $\sigma \in R(G)$. Also, we fix an irreducible cuspidal representation $\rho_0 \in R(GL)$ and an irreducible (essentially) self-contragredient cuspidal representation $\rho \in R(GL)$, such that $\nu^{\alpha} \rho \rtimes \sigma$ reduces for some $\alpha > 0$. We note that $2 \alpha \in \mathbb{Z}$, due to results of \cite{Arthur}, \cite[Th\'{e}or\`{e}me~3.1.1]{Moe2} and \cite[Theorem~7.8]{GanLom}, and that $\nu^{s} \rho \rtimes \sigma$ is irreducible for $s \not\in \{ \alpha, -\alpha \}$.

Let $x$ stand for a half-integer such that $x \geq \alpha$ and $x - \alpha \in \mathbb{Z}$. Then the induced representation
\begin{equation*}
\nu^{-x} \rho \times \nu^{-x+1} \rho \times \cdots \times \nu^{-\alpha} \rho \rtimes \sigma
\end{equation*}
has a unique irreducible subrepresentation, which we denote by $\zeta(\rho, x; \sigma)$. Using \cite[Theorem~3.5]{Matic14}, we deduce that the Aubert dual of $\zeta(\rho, x; \sigma)$ is the unique irreducible subrepresentation of $\nu^{x} \rho \times \nu^{x-1} \rho \times \cdots \times \nu^{\alpha} \rho \rtimes \sigma$. We note that this representation is strongly positive, and will be denoted by $\delta(\rho, x; \sigma)$.

Let $a, b$ denote real numbers such that $b - a \in \mathbb{Z}$. 
We are interested in determining the composition factors of the degenerate principal series 
\begin{equation*}
\zeta([\nu^{-b} \rho_0, \nu^{-a} \rho_0]) \rtimes \zeta(\rho, x; \sigma). 
\end{equation*}
Since in $R(G)$ we have
\begin{align*}
\begin{split}
&\zeta([\nu^{-b} \rho_0, \nu^{-a} \rho_0]) \rtimes \zeta(\rho, x; \sigma) = \zeta([\nu^{a} \widetilde{\rho_0}, \nu^{b} \widetilde{\rho_0}]) \rtimes \zeta(\rho, x; \sigma), \\
& \text{ if } G_n=Sp(2n,F), SO(2n+1,F),\\
& \zeta([\nu^{-b} \rho_0, \nu^{-a} \rho_0]) \rtimes \zeta(\rho, x; \sigma) = \zeta([\nu^{a} \widetilde{\rho_0}\otimes \omega_{\sigma}, \nu^{b} \widetilde{\rho_0}\otimes \omega_{\sigma}]) \rtimes \zeta(\rho, x; \sigma), \\
& \text{ if } G_n=GSpin(2n+1,F),
\end{split}
\end{align*}
we can assume that $-a \leq b$.

By properties of the Aubert involution, the Aubert dual of the degenerate principal series $\zeta([\nu^{-b} \rho_0, \nu^{-a} \rho_0]) \rtimes \zeta(\rho, x; \sigma)$ is the generalized principal series
\begin{align}\label{gosprva}
\begin{split}
   & \delta([\nu^{a} \widetilde{\rho_0}, \nu^{b} \widetilde{\rho_0}]) \rtimes \delta(\rho, x; \sigma), \text{ if } G_n=Sp(2n,F), SO(2n+1,F),\\
   & \delta([\nu^{a} \widetilde{\rho_0}\otimes \omega_{\sigma}, \nu^{b} \widetilde{\rho_0}\otimes \omega_{\sigma}]) \rtimes \delta(\rho, x; \sigma), \text{ if } G_n=GSpin(2n+1,F),
   \end{split}
\end{align}
whose composition factors are completely described in \cite{Mu3} (this has been already noted in \cite[Corollary~4.3]{Jan6}). It follows from \cite[Section~2]{Mu3} (\cite[Proposition 2.5]{Kim1} for $GSpin$ groups) that the induced representation (\ref{gosprva}) is irreducible unless $\rho_0$ is (essentially) self-contragredient. Thus, in what follows we can assume that $\rho_0$ is (essentially) self-contragredient, and let us denote by $\beta$ the unique non-negative real number such that $\nu^{\beta} \rho_0 \rtimes \sigma$ reduces. Again, it follows from \cite[Section~2]{Mu3} that the induced representation (\ref{gosprva}) is irreducible if $a - \beta \not\in \mathbb{Z}$ (In the case of $GSpin$, the argument is similar). So, we can also assume that $a - \beta \in \mathbb{Z}$.

\begin{remark}
\begin{enumerate}
\item[(1)] To work effectively, from now on until Section \ref{Section a1/2}, $G_n$ will only denote $Sp(2n, F)$ and $SO(2n+1, F)$. 
In Section \ref{GSpin}, we will consider the case of $G_n=GSpin(2n+1, F)$. 
\item[(2)] All the lemmas and propositions in the rest of this section are also valid for the odd $GSpin$ case (with same statements, after replacing ``self-contragredient" by ``essentially self-contragredient"), see Section \ref{GSpin} for more detailed comments. 
\end{enumerate}
\end{remark}

We will use the following result \cite[Lemma~5.5]{Jan3} several times.

\begin{lemma} \label{lemajantz}
Suppose that $\pi \in R(G_n)$ is an irreducible representation, $\lambda$ an irreducible representation of the Levi subgroup $M$ of $G_n$, and $\pi$ is a subrepresentation of Ind$_{M}^{G_{n}}(\lambda)$. If $L > M$, then there is an irreducible subquotient $\rho$ of Ind$_{M}^{L}(\lambda)$ such that $\pi$ is a subrepresentation of Ind$_{L}^{G_{n}}(\rho)$.
\end{lemma}

The following result is a direct consequence of \cite[Lemma~2.2]{Matic14}.

\begin{lemma}  \label{lemaprva}
Suppose that the Jacquet module of $\pi$ with respect to the appropriate parabolic subgroup contains an irreducible cuspidal representation of the form $\nu^{a_1} \rho_1 \otimes \nu^{a_2} \rho_2 \otimes \cdots \otimes \nu^{a_k} \rho_k \otimes \sigma$, where $\rho_1, \ldots, \rho_k \in R(GL)$ are self-contragredient representations. Then $\widehat{\pi}$ is a subrepresentation of $\nu^{-a_1} \rho_1 \times \nu^{-a_2} \rho_2 \times \cdots \times \nu^{-a_k} \rho_k \rtimes \sigma$.
\end{lemma}

We will now present a sequence of lemmas which enable us to use an inductive procedure when determining the Aubert duals.

\begin{lemma} \label{lemaindprva}
Let $c$ and $d$ denote positive real numbers such that $d - c$ is a nonnegative integer. Let $\rho_1 \in R(GL)$ denote an irreducible cuspidal self-contragredient representation. If $\pi$ is a subrepresentation of an induced representation of the form $\zeta([\nu^{c} \rho_1, \nu^{d} \rho_1]) \rtimes \pi_1$, where $\pi_1$ is an irreducible representation such that $\mu^{\ast}(\pi_1)$ does not contain an irreducible constituent of the form $\nu^{i} \rho_1 \otimes \pi_2$ for $i \in \{ c, c+1, \ldots, d \}$, with $\pi_2 \in R(G)$, then $\widehat{\pi}$ is the unique irreducible subrepresentation of $\delta([\nu^{-d} \rho_1, \nu^{-c} \rho_1]) \rtimes \widehat{\pi_1}$.
\end{lemma}
\begin{proof}
From properties of the Aubert involution we conclude that $\widehat{\pi}$ is contained in $\delta([\nu^{-d} \rho_1, \nu^{-c} \rho_1]) \rtimes \widehat{\pi_1}$.

From embeddings
\begin{equation*}
\pi \hookrightarrow \zeta([\nu^{c} \rho_1, \nu^{d} \rho_1]) \rtimes \pi_1 \hookrightarrow \nu^{c} \rho_1 \times \cdots \times \nu^{d} \rho_1 \rtimes \pi_1
\end{equation*}
and Frobenius reciprocity, it follows that the Jacquet module of $\pi$ with respect to the appropriate parabolic subgroup contains $\nu^{c} \rho_1 \otimes \cdots \otimes \nu^{d} \rho_1 \otimes \pi_1$.

Using transitivity of Jacquet modules and Lemma \ref{lemaprva}, we obtain that the Jacquet module of $\widehat{\pi}$ with respect to the appropriate parabolic subgroup contains an irreducible constituent of the form $\nu^{-c} \rho_1 \otimes \cdots \otimes \nu^{-d} \rho_1 \otimes \pi'$.

Since $\mu^{\ast}(\pi_1)$ does not contain an irreducible constituent of the form $\nu^{i} \rho_1 \otimes \pi_2$ for $i \in \{ c, c+1, \ldots, d \}$, it follows from Lemma \ref{lemaprva} that $\mu^{\ast}(\widehat{\pi_1})$ does not contain an irreducible constituent of the form $\nu^{-i} \rho_1 \otimes \pi_2$ for $i \in \{ c, c+1, \ldots, d \}$, with $\pi_2 \in R(G)$. Now it follows directly from the structural formula that $\nu^{-c} \rho_1 \otimes \cdots \otimes \nu^{-d} \rho_1 \otimes \widehat{\pi_1}$ is the unique irreducible constituent of the form $\nu^{-c} \rho_1 \otimes \cdots \otimes \nu^{-d} \rho_1 \otimes \pi'$ appearing in the Jacquet module of $\delta([\nu^{-d} \rho_1, \nu^{-c} \rho_1]) \rtimes \widehat{\pi_1}$ with respect to the appropriate parabolic subgroup, and it appears there with multiplicity one. It follows that $\delta([\nu^{-d} \rho_1, \nu^{-c} \rho_1]) \rtimes \widehat{\pi_1}$ contains a unique irreducible subrepresentation.

On the other hand, by Frobenius reciprocity every irreducible subrepresentation of $\delta([\nu^{-d} \rho_1, \nu^{-c} \rho_1]) \rtimes \widehat{\pi_1}$ contains $\nu^{-c} \rho_1 \otimes \cdots \otimes \nu^{-d} \rho_1 \otimes \widehat{\pi_1}$ in the Jacquet module with respect to the appropriate parabolic subgroup. Thus, $\widehat{\pi}$ has to be the unique irreducible subrepresentation of $\delta([\nu^{-d} \rho_1, \nu^{-c} \rho_1]) \rtimes \widehat{\pi_1}$. This proves the lemma.
\end{proof}

For positive integer $m$, real number $t$, and an irreducible cuspidal representation $\rho_1 \in R(GL)$, we denote by $(\nu^{t} \rho_1)^m$ the induced representation $\nu^{t} \rho_1 \times \cdots \times \nu^{t} \rho_1$, where $\nu^{t} \rho_1$ appears $m$ times. Note that the induced representation $\zeta([\nu^{c} \rho_1, \nu^{d} \rho_1]) \times (\nu^{t} \rho_1)^m$ is irreducible for $t \in \{ c, c+1, \ldots, d \}$ \cite{Zel}. In the same way as in the proof of Lemma \ref{lemaindprva}, one obtains the following results.

\begin{lemma} \label{lemainddruga}
Let $c$ and $d$ denote positive real numbers such that $d - c$ is a nonnegative integer. Let $\rho_1 \in R(GL)$ denote an irreducible cuspidal self-contragredient representation. Suppose that $\pi$ is a subrepresentation of an induced representation of the form $\zeta([\nu^{c} \rho_1, \nu^{d} \rho_1]) \times (\nu^{t} \rho_1)^m \rtimes \pi_1$, where $t \in \{ c, c+1, \ldots, d \}$, $\pi_1$ is  irreducible and $\mu^{\ast}(\pi_1)$ does not contain an irreducible constituent of the form $\nu^{i} \rho_1 \otimes \pi_2$ for $i \in \{ c, c+1, \ldots, d \}$, with $\pi_2 \in R(G)$. Then $\widehat{\pi}$ is the unique irreducible subrepresentation of $\delta([\nu^{-d} \rho_1, \nu^{-c} \rho_1]) \times (\nu^{-t} \rho_1)^m \rtimes \widehat{\pi_1}$.
\end{lemma}

\begin{lemma} \label{lemaindtreca}
Let $c$ and $d$ denote positive real numbers such that $d - c$ is a nonnegative integer. Let $\rho_1 \in R(GL)$ denote an irreducible cuspidal self-contragredient representation. Suppose that $\pi$ is a subrepresentation of an induced representation of the form $\zeta([\nu^{c} \rho_1, \nu^{d} \rho_1]) \times (\nu^{d} \rho_1)^m \rtimes \pi_1$, where $\pi_1$ is an irreducible representation such that the Jacquet module of $\pi_1$ with respect to the appropriate parabolic subgroup does not contain an irreducible constituent of the form $\nu^{d-k} \rho_1 \otimes \cdots \otimes \nu^{d-1} \rho_1 \otimes \nu^{d} \rho_1 \otimes \pi'$ for a nonnegative integer $k < d$, with $\pi' \in R(G)$. Then $\widehat{\pi}$ is the unique irreducible subrepresentation of $\delta([\nu^{-d} \rho_1, \nu^{-c} \rho_1]) \times (\nu^{-d} \rho_1)^m \rtimes \widehat{\pi_1}$.
\end{lemma}

\begin{lemma} \label{lemarazl}
Suppose that $\rho_0 \not\cong \rho$ and let $\pi$ denote an irreducible subquotient of $\delta([\nu^{a} \rho_0, \nu^{b} \rho_0]) \rtimes \delta(\rho, x; \sigma)$. Then there is an irreducible representation $\pi_1 \in R(G)$ such that $\pi$ is a subrepresentation of $\delta([\nu^{\alpha} \rho, \nu^{x} \rho]) \rtimes \pi_1$ and $\widehat{\pi}$ is the unique irreducible subrepresentation of $\nu^{-x} \rho \times \nu^{-x+1} \rho \times \cdots \times \nu^{-\alpha} \rho \rtimes \widehat{\pi_1}$. Furthermore, if $\widehat{\pi_1} \cong L(\delta_1, \delta_2, \ldots, \delta_k; \tau_{temp})$, where $e(\delta_i) \leq e(\delta_j)$ for $i \leq j$, then
\begin{equation*}
\widehat{\pi} \cong L(\nu^{-x} \rho, \nu^{-x+1} \rho, \ldots, \nu^{-\alpha} \rho, \delta_1, \delta_2, \ldots, \delta_k; \tau_{temp}).
\end{equation*}
\end{lemma}
\begin{proof}
By the results of \cite{Mu3}, there is an irreducible tempered representation $\tau \in R(G)$ such that either $\pi \cong \tau$ or $\pi \cong L(\delta([\nu^{c} \rho_0, \nu^{-a} \rho_0]); \tau)$, for some $c \geq -b$ such that $c - a < 0$. Also, it is easy to see that there is an irreducible representation $\tau_1$ such that $\tau$ is a subrepresentation of $\delta([\nu^{\alpha} \rho, \nu^{x} \rho]) \rtimes \tau_1$, and there are no twists of $\rho$ appearing in the cuspidal support of $\tau_1$. If $\pi \cong \tau$, we can take $\pi_1 \cong \tau_1$. Otherwise, since $\rho_0 \not\cong \rho$ we have
\begin{align*}
\pi & \hookrightarrow \delta([\nu^{c} \rho_0, \nu^{-a} \rho_0]) \rtimes \tau \hookrightarrow \delta([\nu^{c} \rho_0, \nu^{-a} \rho_0]) \times \delta([\nu^{\alpha} \rho, \nu^{x} \rho]) \rtimes \tau_1 \\
& \cong \delta([\nu^{\alpha} \rho, \nu^{x} \rho]) \times \delta([\nu^{c} \rho_0, \nu^{-a} \rho_0]) \rtimes \tau_1,
\end{align*}
and by \cite[Lemma~3.2]{MT1} there is an irreducible representation $\pi_1$ such that $\pi$ is a subrepresentation of $\delta([\nu^{\alpha} \rho, \nu^{x} \rho]) \rtimes \pi_1$. Since there are no twists of $\rho$ appearing in the cuspidal support of $\pi_1$, it can be seen in the same way as in the proof of Lemma \ref{lemaindprva} that $\widehat{\pi}$ is the unique irreducible subrepresentation of $\nu^{-x} \rho \times \nu^{-x+1} \rho \times \cdots \times \nu^{-\alpha} \rho \rtimes \widehat{\pi_1}$.

If we write $\widehat{\pi_1} \cong L(\delta_1, \delta_2, \ldots, \delta_k; \tau_{temp})$, then $\delta_i \cong \delta([\nu^{x_i} \rho_0, \nu^{y_i} \rho_0])$ for $i=1, 2, \ldots, k$, and we have $\nu^{z} \rho \times \delta_i \cong \delta_i \times \nu^{z} \rho$ for all $i = 1, 2, \ldots, k$ and $z \in \mathbb{R}$. This ends the proof.
\end{proof}

The following result provides embeddings needed for an inductive determination of the Aubert duals.

\begin{proposition}  \label{propulag}
Let $\rho_1 \in R(GL)$ denote an irreducible self-contragredient cuspidal representation, and let $\sigma_{sp} \in R(G)$ denote a strongly positive discrete series. Let $k, l$ denote half-integers such that $k - l$ is a positive integer and $k + l > 0$.
\begin{enumerate}[(1)]
\item If $\nu^{k} \rho_1 \rtimes \sigma_{sp}$ is irreducible and $k \geq -l + 2$, then $L(\delta([\nu^{-k} \rho_1, \nu^{-l} \rho_1]); \sigma_{sp})$ is a subrepresentation of $\nu^{k} \rho_1 \rtimes L(\delta([\nu^{-k+1} \rho_1, \nu^{-l} \rho_1]); \sigma_{sp})$.
\item If $\mu^{\ast}(\sigma_{sp})$ does not contain an irreducible constituent of the form $\nu^{-l} \rho_1 \otimes \pi$, with $\pi \in R(G)$, then $L(\delta([\nu^{-k} \rho_1, \nu^{-l} \rho_1]); \sigma_{sp})$ is a subrepresentation of $\nu^{-l} \rho_1 \rtimes L(\delta([\nu^{-k} \rho_1$, $\nu^{-l-1} \rho_1]); \sigma_{sp})$.
\item Suppose that $\sigma_{sp}$ is a subrepresentation of $\nu^{t} \rho_1 \rtimes \sigma'_{sp}$ for some $t \neq k$, $t \neq -l+1 $ and a strongly positive representation $\sigma'_{sp}$. Then $L(\delta([\nu^{-k} \rho_1, \nu^{-l} \rho_1]);$ $\sigma_{sp})$ is a subrepresentation of $\nu^{t} \rho_1 \rtimes L(\delta([\nu^{-k} \rho_1, \nu^{-l} \rho_1]); \sigma'_{sp})$.
\end{enumerate}
\end{proposition}
\begin{proof}
We only prove the first part of the proposition, other parts can be proved in the same way but more easily. We have the following embeddings and isomorphisms:
\begin{align*}
L(\delta([\nu^{-k} \rho_1, \nu^{-l} \rho_1]); \sigma_{sp}) & \hookrightarrow \delta([\nu^{-k} \rho_1, \nu^{-l} \rho_1]) \rtimes \sigma_{sp}\\
& \hookrightarrow \delta([\nu^{-k+1} \rho_1, \nu^{-l} \rho_1]) \times \nu^{-k} \rho_1 \rtimes \sigma_{sp} \\
& \cong \delta([\nu^{-k+1} \rho_1, \nu^{-l} \rho_1]) \times \nu^{k} \rho_1 \rtimes \sigma_{sp} \\
& \cong \nu^{k} \rho_1 \times \delta([\nu^{-k+1} \rho_1, \nu^{-l} \rho_1]) \rtimes \sigma_{sp}.
\end{align*}
By Lemma \ref{lemajantz}, there is an irreducible subquotient $\pi$ of $\delta([\nu^{-k+1} \rho_1, \nu^{-l} \rho_1]) \rtimes \sigma_{sp}$ such that $L(\delta([\nu^{-k} \rho_1, \nu^{-l} \rho_1]); \sigma_{sp})$ is a subrepresentation of $\nu^{k} \rho_1 \rtimes \pi$. Frobenius reciprocity implies that $\mu^{\ast}(\nu^{k} \rho_1 \rtimes \pi)$ contains $\delta([\nu^{-k} \rho_1, \nu^{-l} \rho_1]) \otimes \sigma_{sp}$.

Using the structural formula and a description of the Jacquet modules of strongly positive representations, provided in \cite[Theorem~4.6]{Matic4} and \cite[Section~7]{MatTad}, we deduce that $\mu^{\ast}(\delta([\nu^{-k+1} \rho_1, \nu^{-l} \rho_1]) \rtimes \sigma_{sp})$ does not contain an irreducible constituent of the form $\delta([\nu^{-k} \rho_1, \nu^{-l} \rho_1]) \otimes \pi_1$, with $\pi_1 \in R(G)$. Thus, $\mu^{\ast}(\pi)$ contains $\delta([\nu^{-k+1} \rho_1, \nu^{-l} \rho_1]) \otimes \sigma_{sp}$ and it
is a direct consequence of the Langlands classification that $\pi \cong L(\delta([\nu^{-k+1} \rho_1, \nu^{-l} \rho_1]); \sigma_{sp})$.
\end{proof}

Note that both description of subquotients of $\delta([\nu^{a}\rho_0, \nu^{b}\rho_0]) \rtimes \delta(\rho, x; \sigma)$ and their Aubert duals depend on the reduciblity point $\beta$ of $\rho_0$ and $\sigma$ \cite{Matic10, Mu3}. Description of the Aubert duals happens to be slightly different in the case $\beta = 0$. Accordingly we also consider two cases: Section \ref{Section beta0} is the case $\beta=0$ (Section 5 of \cite{Matic10}) and Section \ref{Section a1}, \ref{Section a0}, \ref{Section a1/2} is the case $\beta>0$ (Section 4 of \cite{Matic10}).

\section{Case $\beta = 0$}\label{Section beta0}
In this section we consider the $\beta = 0$ case. Note that this implies $a \in \mathbb{Z}$.

The following irreducibility result is a direct consequence of \cite[Proposition~3.1]{Mu3}.

\begin{proposition}
Degenerate principal series $\zeta([\nu^{-b} \rho_0, \nu^{-a} \rho_0]) \rtimes \zeta(\rho, x; \sigma)$ is irreducible if and only if $a \geq 1$. 
\end{proposition}

We consider the remaining cases in the following proposition.

\begin{proposition} \label{propnulaprva}
Suppose that $a \leq 0$, and write $\rho_0 \rtimes \sigma = \tau_1 + \tau_{-1}$, as a sum of mutually non-isomorphic irreducible tempered representations. If $-a < b$, then in $R(G)$ we have:
\begin{gather*}
\zeta([\nu^{-b} \rho_0, \nu^{-a} \rho_0]) \rtimes \zeta(\rho, x; \sigma)  = \\
L(\nu^{-x} \rho, \ldots, \nu^{-\alpha} \rho, \nu^{-b} \rho_0, \ldots, \nu^{a-1} \rho_0, \nu^{a} \rho_0, \nu^{a} \rho_0, \ldots, \nu^{-1} \rho_0, \nu^{-1} \rho_0, \tau_1) + \\
L(\nu^{-x} \rho, \ldots, \nu^{-\alpha} \rho, \nu^{-b} \rho_0, \ldots, \nu^{a-1} \rho_0, \nu^{a} \rho_0, \nu^{a} \rho_0, \ldots, \nu^{-1} \rho_0, \nu^{-1} \rho_0, \tau_{-1}) + \\
L(\nu^{-x} \rho, \ldots, \nu^{-\alpha} \rho, \nu^{-b} \rho_0, \ldots, \nu^{a-2} \rho_0, \delta([\nu^{a-1} \rho_0, \nu^{a} \rho_0]), \ldots, \delta([\nu^{-1} \rho_0, \rho_0]; \sigma)).
\end{gather*}
If $-a = b$, then in $R(G)$ we have:
\begin{gather*}
\zeta([\nu^{a} \rho_0, \nu^{-a} \rho_0]) \rtimes \zeta(\rho, x; \sigma)  = \\
L(\nu^{-x} \rho, \ldots, \nu^{-\alpha} \rho, \nu^{a} \rho_0, \nu^{a} \rho_0, \ldots, \nu^{-1} \rho_0, \nu^{-1} \rho_0, \tau_1) + \\
L(\nu^{-x} \rho, \ldots, \nu^{-\alpha} \rho, \nu^{a} \rho_0, \nu^{a} \rho_0, \ldots, \nu^{-1} \rho_0, \nu^{-1} \rho_0, \tau_{-1}).
\end{gather*}
\end{proposition}
\begin{proof}
We will only comment the case $-a < b$, since the case $-a = b$ can be handled in the same way as in the proof of \cite[Theorem~5.1]{Matic10}. By \cite[Theorem~2.1]{Mu3} and classification of discrete series \cite{KimMatic, MT1}, in $R(G)$ we have
\begin{equation*}
\delta([\nu^{a} \rho_0, \nu^{b} \rho_0]) \rtimes \delta(\rho, x; \sigma) = \sigma_1 + \sigma_{-1} + L(\delta([\nu^{-b} \rho_0, \nu^{-a} \rho_0]); \delta(\rho, x; \sigma)),   
\end{equation*}
where $\sigma_i$ is a discrete series subrepresentation of $\delta([\nu^{a} \rho_0, \nu^{b} \rho_0]) \rtimes \delta(\rho, x; \sigma)$ such that
\begin{equation*}
\mu^{\ast}(\sigma_i) \geq \delta([\nu \rho_0, \nu^{-a} \rho_0]) \times \delta([\nu \rho_0, \nu^{b} \rho_0]) \times \delta([\nu^{\alpha} \rho, \nu^{x} \rho]) \otimes \tau_i
\end{equation*}
and
\begin{equation*}
\mu^{\ast}(\sigma_i) \not\geq \delta([\nu \rho_0, \nu^{-a} \rho_0]) \times \delta([\nu \rho_0, \nu^{b} \rho_0]) \times \delta([\nu^{\alpha} \rho, \nu^{x} \rho]) \otimes \tau_{-i},
\end{equation*}
for $i \in \{ 1, -1 \}$.

Since $\sigma_i$ is a subrepresentation of $\delta([\nu^{a} \rho_0, \nu^{b} \rho_0]) \rtimes \delta(\rho, x; \sigma)$, for $i \in \{ 1, -1 \}$, we have
\begin{align*}
\sigma_i & \hookrightarrow \delta([\nu^{a} \rho_0, \nu^{b} \rho_0]) \rtimes \delta(\rho, x; \sigma) \hookrightarrow \delta([\nu^{a} \rho_0, \nu^{b} \rho_0]) \times \delta([\nu^{\alpha} \rho, \nu^{x} \rho]) \rtimes \sigma \\
& \cong \delta([\nu^{\alpha} \rho, \nu^{x} \rho]) \times \delta([\nu^{a} \rho_0, \nu^{b} \rho_0]) \rtimes \sigma.
\end{align*}
By Lemma \ref{lemajantz}, there is an irreducible subquotient $\pi_i$ of $\delta([\nu^{a} \rho_0, \nu^{b} \rho_0]) \rtimes \sigma$ such that $\sigma_i$ is a subrepresentation of $\delta([\nu^{\alpha} \rho, \nu^{x} \rho]) \rtimes \pi_i$.

Using \cite[Theorem~2.1]{Mu3} and classification of discrete series one more time, we obtain that in $R(G)$ we have
\begin{equation*}
\delta([\nu^{a} \rho_0, \nu^{b} \rho_0]) \rtimes \sigma = \sigma'_1 + \sigma'_{-1} + L(\delta([\nu^{-b} \rho_0, \nu^{-a} \rho_0]); \sigma),
\end{equation*}
where $\sigma'_i$ is a discrete series subrepresentation of $\delta([\nu^{a} \rho_0, \nu^{b} \rho_0]) \rtimes \sigma$ such that $\mu^{\ast}(\sigma'_i) \geq \delta([\nu \rho_0, \nu^{-a} \rho_0]) \times \delta([\nu \rho_0, \nu^{b} \rho_0]) \otimes \tau_i$ and $\mu^{\ast}(\sigma'_i) \not\geq \delta([\nu \rho_0, \nu^{-a} \rho_0]) \times \delta([\nu \rho_0, \nu^{b} \rho_0]) \otimes \tau_{-i}$, for $i \in \{ 1, -1 \}$. Also, note that $\mu^{\ast}(L(\delta([\nu^{-b} \rho_0, \nu^{-a} \rho_0]); \sigma))$ does not contain $\delta([\nu^{a} \rho_0, \nu^{b} \rho_0]) \otimes \sigma$, since both $\mu^{\ast}(\sigma'_1)$ and $\mu^{\ast}(\sigma'_{-1})$ contain $\delta([\nu^{a} \rho_0, \nu^{b} \rho_0]) \otimes \sigma$, and $\mu^{\ast}(\delta([\nu^{a} \rho_0, \nu^{b} \rho_0]) \rtimes \sigma)$ contains $\delta([\nu^{a} \rho_0, \nu^{b} \rho_0]) \otimes \sigma$ with multiplicity two.

Thus, $\pi_i \cong \sigma'_i$. Now Lemma \ref{lemarazl} and \cite[Theorem~5.1]{Matic10} imply that
\begin{equation*}
\widehat{\sigma_i} \cong L(\nu^{-x} \rho, \ldots, \nu^{-\alpha} \rho, \nu^{-b} \rho_0, \ldots, \nu^{a-1} \rho_0, \nu^{a} \rho_0, \nu^{a} \rho_0, \ldots, \nu^{-1} \rho_0, \nu^{-1} \rho_0, \tau_{-i}).
\end{equation*}
In the same way we obtain that $L(\delta([\nu^{-b} \rho_0, \nu^{-a} \rho_0]); \delta(\rho, x; \sigma))$ is a subrepresentation of $\delta([\nu^{\alpha} \rho, \nu^{x} \rho]) \rtimes L(\delta([\nu^{-b} \rho_0, \nu^{-a} \rho_0]); \sigma)$. By Lemma \ref{lemarazl}, it remains to determine the Aubert dual of $L(\delta([\nu^{-b} \rho_0, \nu^{-a} \rho_0]); \sigma)$. Since $b > 0$, if $b \geq -a+2$ then using the first part of Proposition \ref{propulag} we get that $L(\delta([\nu^{-b} \rho_0, \nu^{-a} \rho_0]); \sigma)$ is a subrepresentation of $\nu^{b} \rho_0 \rtimes L(\delta([\nu^{-b+1} \rho_0, \nu^{-a} \rho_0]); \sigma)$. Also, it follows from the structural formula that $\mu^{\ast}(L(\delta([\nu^{-b+1} \rho_0, \nu^{-a} \rho_0]); \sigma))$ does not contain an irreducible constituent of the form $\nu^{b} \rho_0 \otimes \pi'$. Using Lemma \ref{lemaindprva} and repeating this procedure, we deduce that the Aubert dual of $L(\delta([\nu^{-b} \rho_0, \nu^{-a} \rho_0]); \sigma)$ is an irreducible subrepresentation of
\begin{equation*}
\nu^{-b} \rho_0 \times \cdots \times \nu^{a-2} \rho_0 \rtimes \widehat{L(\delta([\nu^{a-1} \rho_0, \nu^{-a} \rho_0]); \sigma)}.
\end{equation*}
The representation $L(\delta([\nu^{a-1} \rho_0, \nu^{-a} \rho_0]); \sigma)$ is the unique irreducible quotient of the induced representation $\delta([\nu^{a} \rho_0, \nu^{-a+1} \rho_0]) \rtimes \sigma$. By \cite[Theorem~2.1]{Mu3}, $\delta([\nu^{a} \rho_0, \nu^{-a+1} \rho_0]) \rtimes \sigma$ contains two irreducible subrepresentations and Frobenius reciprocity implies that each of them contains an irreducible constituent of the form $\nu^{-a+1} \rho_0 \otimes \pi$ in the Jacquet module with respect to the appropriate parabolic subgroup.

If $\nu^{-a+1} \rho_0 \otimes \pi$ is an irreducible constituent of $\mu^{\ast}(\delta([\nu^{a} \rho_0, \nu^{-a+1} \rho_0]) \rtimes \sigma)$, it follows from the structural formula that $\pi$ is an irreducible subquotient of $\delta([\nu^{a} \rho_0, \nu^{-a} \rho_0]) \rtimes \sigma$, which is a length two representation. Thus, there are only two irreducible constituents of the form $\nu^{-a+1} \rho_0 \otimes \pi$ appearing $\mu^{\ast}(\delta([\nu^{a} \rho_0, \nu^{-a+1} \rho_0]) \rtimes \sigma)$, and $\mu^{\ast}(L(\delta([\nu^{a-1} \rho_0, \nu^{-a} \rho_0]); \sigma))$ does not contain any of them.

From the second part of Proposition \ref{propulag} follows that $L(\delta([\nu^{a-1} \rho_0, \nu^{-a} \rho_0]); \sigma)$ is a subrepresentation of $\nu^{-a} \rho_0 \rtimes L(\delta([\nu^{a-1} \rho_0, \nu^{-a-1} \rho_0]); \sigma)$.

Since $a-1 \leq -1$, using the first part of Proposition \ref{propulag} we also obtain
\begin{equation*}
L(\delta([\nu^{a-1} \rho_0, \nu^{-a-1} \rho_0]); \sigma) \hookrightarrow \nu^{-a+1} \rho_0 \rtimes L(\delta([\nu^{a} \rho_0, \nu^{-a-1} \rho_0]); \sigma).
\end{equation*}

Consequently, $L(\delta([\nu^{a-1} \rho_0, \nu^{-a} \rho_0]); \sigma)$ is a subrepresentation of
\begin{equation*}
\nu^{-a} \rho_0 \times \nu^{-a+1} \rho_0 \rtimes L(\delta([\nu^{a} \rho_0, \nu^{-a-1} \rho_0]); \sigma),
\end{equation*}
and there is an irreducible subquotient $\pi_2$ of $\nu^{-a} \rho_0 \times \nu^{-a+1} \rho_0$ such that $L(\delta([\nu^{a-1} \rho_0, \nu^{-a} \rho_0]); \sigma)$ is a subrepresentation of $\pi_2 \rtimes L(\delta([\nu^{a} \rho_0, \nu^{-a-1} \rho_0]); \sigma)$. Since $\mu^{\ast}(L(\delta([\nu^{a-1} \rho_0, \nu^{-a} \rho_0]); \sigma))$ does not contain an irreducible constituent of the form $\nu^{-a-1} \rho_0 \otimes \pi'$, it follows that $\pi_2 \not\cong \delta([\nu^{-a} \rho_0, \nu^{-a+1} \rho_0])$, so we have that $\pi_2 \cong \zeta([\nu^{-a} \rho_0, \nu^{-a+1} \rho_0])$. It can also be seen, following the same arguments as for $L(\delta([\nu^{a-1} \rho_0, \nu^{-a} \rho_0]); \sigma)$, that $\mu^{\ast}(L(\delta([\nu^{a} \rho_0, \nu^{-a-1} \rho_0]); \sigma))$ does not contain an irreducible constituents of the form $\nu^{i} \rho_0 \otimes \pi'$, for $i \in \{ -a + 1, -a \}$. Now Lemma \ref{lemaindprva} implies that $\widehat{L(\delta([\nu^{a-1} \rho_0, \nu^{-a} \rho_0]); \sigma)}$ is the unique irreducible subrepresentation of $\delta([\nu^{a-1} \rho_0, \nu^{a} \rho_0]) \rtimes \widehat{L(\delta([\nu^{a} \rho_0, \nu^{-a-1} \rho_0]); \sigma)}$, and a repeated application of this procedure ends the proof.
\end{proof}

\section{Case $a \geq 1$}\label{Section a1}

From now on, we assume that $\beta > 0$. In this section we consider the case $a \geq 1$. Let us first consider the more complicated case $\rho_0 \cong \rho$. Directly from \cite[Proposition~3.1]{Mu3} we obtain the following reducibility criterion:

\begin{proposition}
Degenerate principal series $\zeta([\nu^{-b} \rho, \nu^{-a} \rho]) \rtimes \zeta(\rho, x; \sigma)$ reduces if and only if one of the following holds:
\begin{itemize}
\item $a \leq \alpha - 1 \leq b < x$,
\item $a \leq x+1$ and $x < b$.
\end{itemize}
\end{proposition}


\begin{proposition}  \label{propspprva}
If $a \leq \alpha - 1 \leq b < x$, then in $R(G)$ we have
\begin{gather*}
\zeta([\nu^{-b} \rho, \nu^{-a} \rho]) \rtimes \zeta(\rho, x; \sigma) \cong \\ L(\nu^{-x} \rho, \ldots, \nu^{-b-1} \rho, \nu^{-b} \rho, \nu^{-b} \rho, \ldots, \nu^{-\alpha} \rho, \nu^{-\alpha} \rho, \nu^{-\alpha+1} \rho, \ldots, \nu^{-a} \rho; \sigma) + \\
L(\nu^{-x} \rho, \ldots, \nu^{-b-2} \rho, \delta([\nu^{-b-1} \rho, \nu^{-b} \rho]), \ldots, \delta([\nu^{-\alpha} \rho, \nu^{-\alpha+1} \rho]), \nu^{-\alpha+2} \rho, \ldots, \nu^{-a} \rho; \sigma).
\end{gather*}
\end{proposition}
\begin{proof}
In $R(G)$ we have
\begin{gather*}
\delta([\nu^{a} \rho, \nu^{b} \rho]) \rtimes \delta(\rho, x; \sigma) = L(\delta([\nu^{-b} \rho, \nu^{-a} \rho]); \delta(\rho, x; \sigma)) + \\
L(\delta([\nu^{-\alpha+2} \rho, \nu^{-a} \rho]); \sigma_{sp}),
\end{gather*}
where $\sigma_{sp}$ is the unique irreducible subrepresentation of $\delta([\nu^{\alpha-1} \rho, \nu^{b} \rho]) \rtimes \delta(\rho, x; \sigma)$. We note that $\sigma_{sp}$ is a strongly positive discrete series.

Let us first determine the Aubert dual of $L(\delta([\nu^{-b} \rho, \nu^{-a} \rho]); \delta(\rho, x; \sigma))$. The third part of Proposition \ref{propulag} implies that
\begin{equation*}
L(\delta([\nu^{-b} \rho, \nu^{-a} \rho]); \delta(\rho, x; \sigma)) \hookrightarrow \nu^{x} \rho \rtimes L(\delta([\nu^{-b} \rho, \nu^{-a} \rho]); \delta(\rho, x-1; \sigma)).
\end{equation*}
Using the structural formula and a description of the Jacquet modules of strongly positive representations, we deduce that $\mu^{\ast}(\delta([\nu^{-b} \rho, \nu^{-a} \rho]) \rtimes \delta(\rho, x-1; \sigma))$ does not contain an irreducible constituent of the form $\nu^{x} \rho \otimes \pi_2$. Repeating this procedure and using Lemma \ref{lemaindprva}, we obtain that the Aubert dual of $L(\delta([\nu^{-b} \rho, \nu^{-a} \rho]); \delta(\rho, x; \sigma))$ is an irreducible subrepresentation of
\begin{equation*}
\nu^{-x} \rho \times \nu^{-x+1} \rho \times \cdots \times \nu^{-b-1} \rho \rtimes \widehat{L(\delta([\nu^{-b} \rho, \nu^{-a} \rho]); \delta(\rho, b; \sigma))}.
\end{equation*}
Since $\nu^{b} \rho \rtimes \delta(\rho, b; \sigma)$ is irreducible, by \cite[Proposition~3.1]{Mu3}, we have
\begin{align*}
L(\delta([\nu^{-b} \rho, \nu^{-a} \rho]); \delta(\rho, b; \sigma)) &\hookrightarrow \delta([\nu^{-b+1} \rho, \nu^{-a} \rho]) \times \nu^{-b} \rho \rtimes \delta(\rho, b; \sigma) \\
&\cong \delta([\nu^{-b+1} \rho, \nu^{-a} \rho]) \times \nu^{b} \rho \rtimes \delta(\rho, b; \sigma) \\
&\hookrightarrow \delta([\nu^{-b+1} \rho, \nu^{-a} \rho]) \times \nu^{b} \rho \times \nu^{b} \rho \rtimes \delta(\rho, b-1; \sigma) \\
&\cong \nu^{b} \rho \times \nu^{b} \rho \times \delta([\nu^{-b+1} \rho, \nu^{-a} \rho]) \rtimes \delta(\rho, b-1; \sigma).
\end{align*}
Note that $\delta([\nu^{-b+1} \rho, \nu^{-a} \rho]) \rtimes \delta(\rho, b-1; \sigma)$ is irreducible, thus isomorphic to $L(\delta([\nu^{-b+1} \rho, \nu^{-a} \rho]); \delta(\rho, b-1; \sigma))$ and that $\mu^{\ast}(\delta([\nu^{-b+1} \rho, \nu^{-a} \rho]) \rtimes \delta(\rho, b-1; \sigma))$ does not contain an irreducible constituent of the form $\nu^{b} \rho \otimes \pi$. A repeated application of Lemma \ref{lemainddruga} and the previous procedure implies that the Aubert dual of $L(\delta([\nu^{-b} \rho, \nu^{-a} \rho]); \delta(\rho, b; \sigma))$ is an irreducible subrepresentation of
\begin{equation*}
\nu^{-b} \rho \times \nu^{-b} \rho \times \cdots \times \nu^{-\alpha} \rho \times \nu^{-\alpha} \rho \rtimes \widehat{L(\delta([\nu^{-\alpha+1} \rho, \nu^{-a} \rho]); \sigma)}.
\end{equation*}
Since the induced representation $\delta([\nu^{-\alpha+1} \rho, \nu^{-a} \rho]) \rtimes \sigma$ is also irreducible, its Jacquet module with respect to the appropriate parabolic subgroup contains $\nu^{\alpha - 1} \rho \otimes \cdots \otimes \nu^{a} \rho \otimes \sigma$. Now Lemma \ref{lemaprva} implies that the Aubert dual of $L(\delta([\nu^{-\alpha+1} \rho, \nu^{-a} \rho]); \sigma)$ is the unique irreducible subrepresentation of $\nu^{-\alpha + 1} \rho \times \cdots \times \nu^{-a} \rho \rtimes \sigma$. Altogether, the Aubert dual of $L(\delta([\nu^{-b} \rho, \nu^{-a} \rho]); \delta(\rho, x; \sigma))$ is isomorphic to
\begin{equation*}
L(\nu^{-x} \rho, \ldots, \nu^{-b-1} \rho, \nu^{-b} \rho, \nu^{-b} \rho, \ldots, \nu^{-\alpha} \rho, \nu^{-\alpha} \rho, \nu^{-\alpha+1} \rho, \ldots, \nu^{-a} \rho; \sigma).
\end{equation*}
It remains to determine the Aubert dual of $L(\delta([\nu^{-\alpha+2} \rho, \nu^{-a} \rho]); \sigma_{sp})$.

If $x > b + 1$, it follows from \cite[Section~3]{Matic4} that $\sigma_{sp}$ is a subrepresentation of $\nu^{x} \rho \rtimes \sigma'_{sp}$, where $\sigma'_{sp}$ is the unique irreducible subrepresentation of $\delta([\nu^{\alpha-1} \rho, \nu^{b} \rho]) \rtimes \delta(\rho, x-1; \sigma)$. The third part of Proposition \ref{propulag} implies that $L(\delta([\nu^{-\alpha+2} \rho, \nu^{-a} \rho]); \sigma_{sp})$ is a subrepresentation of $\nu^{x} \rho \rtimes L(\delta([\nu^{-\alpha+2} \rho, \nu^{-a} \rho]); \sigma'_{sp})$. Using Lemma \ref{lemaindprva} and continuing in the same way, we deduce that the Aubert dual of $L(\delta([\nu^{-\alpha+2} \rho, \nu^{-a} \rho]); \sigma_{sp})$ is a subrepresentation of
\begin{equation*}
\nu^{-x} \rho \times \cdots \times \nu^{-b-2} \rho \rtimes \widehat{L(\delta([\nu^{-\alpha+2} \rho, \nu^{-a} \rho]); \sigma^{(1)}_{sp})},
\end{equation*}
where $\sigma^{(1)}_{sp}$ is the unique irreducible subrepresentation of $\delta([\nu^{\alpha-1} \rho, \nu^{b} \rho]) \rtimes \delta(\rho, b+1; \sigma)$. From embeddings of strongly positive representations (\cite[Section~3]{Matic4}), using Proposition \ref{propulag} $(3)$ twice, we get
\begin{equation*}
L(\delta([\nu^{-\alpha+2} \rho, \nu^{-a} \rho]); \sigma^{(1)}_{sp}) \hookrightarrow \nu^{b} \rho \times \nu^{b+1} \rho \rtimes L(\delta([\nu^{-\alpha+2} \rho, \nu^{-a} \rho]); \sigma^{(2)}_{sp}),
\end{equation*}
where $\sigma^{(2)}_{sp}$ is the unique irreducible subrepresentation of $\delta([\nu^{\alpha-1} \rho, \nu^{b-1} \rho]) \rtimes \delta(\rho, b; \sigma)$. Now \cite[Theorem~3.4]{Matic4} implies
\begin{equation*}
L(\delta([\nu^{-\alpha+2} \rho, \nu^{-a} \rho]); \sigma^{(1)}_{sp}) \hookrightarrow \zeta([\nu^{b} \rho, \nu^{b+1} \rho]) \rtimes L(\delta([\nu^{-\alpha+2} \rho, \nu^{-a} \rho]); \sigma^{(2)}_{sp}).
\end{equation*}
Using a repeated application of Lemma \ref{lemainddruga} and continuing in the same way, we obtain that the Aubert dual of $L(\delta([\nu^{-\alpha+2} \rho, \nu^{-a} \rho]); \sigma^{(1)}_{sp})$ is a subrepresentation of
\begin{equation*}
\delta([\nu^{-b-1} \rho, \nu^{-b} \rho])\times \cdots\times \delta([\nu^{-\alpha} \rho, \nu^{-\alpha+1} \rho]) \rtimes \widehat{L(\delta([\nu^{-\alpha+2} \rho, \nu^{-a} \rho]); \sigma)},
\end{equation*}
and it can be seen in the same way as in the case of $L(\delta([\nu^{-\alpha+1} \rho, \nu^{-a} \rho]); \sigma)$ that the Aubert dual of $L(\delta([\nu^{-\alpha+2} \rho, \nu^{-a} \rho]); \sigma)$ is the unique irreducible subrepresentation of $\nu^{-\alpha + 2} \rho \times \cdots \times \nu^{-a} \rho \rtimes \sigma$. This ends the proof.
\end{proof}

\begin{proposition} \label{propvod}
Suppose that $a \leq x+1$ and $x < b$. If $a > \alpha$, then in $R(G)$ we have
\begin{gather*}
\zeta([\nu^{-b} \rho, \nu^{-a} \rho]) \rtimes \zeta(\rho, x; \sigma) = \\ L(\nu^{-b} \rho, \ldots, \nu^{-x-1} \rho, \nu^{-x} \rho, \nu^{-x} \rho, \ldots, \nu^{-a} \rho, \nu^{-a} \rho, \nu^{-a+1} \rho, \ldots, \nu^{-\alpha} \rho; \sigma) + \\
L(\nu^{-b} \rho, \ldots, \nu^{-x-2} \rho, \delta([\nu^{-x-1} \rho, \nu^{-x} \rho]), \ldots, \delta([\nu^{-a} \rho, \nu^{-a+1} \rho]), \nu^{-a+2} \rho, \ldots, \nu^{-\alpha} \rho; \sigma).
\end{gather*}
If $a \leq \alpha$, then in $R(G)$ we have
\begin{gather*}
\zeta([\nu^{-b} \rho, \nu^{-a} \rho]) \rtimes \zeta(\rho, x; \sigma) = \\ L(\nu^{-b} \rho, \ldots, \nu^{-x-1} \rho, \nu^{-x} \rho, \nu^{-x} \rho, \ldots, \nu^{-\alpha} \rho, \nu^{-\alpha} \rho, \nu^{-\alpha+1} \rho, \ldots, \nu^{-a} \rho; \sigma) + \\
L(\nu^{-b} \rho, \ldots, \nu^{-x-2} \rho, \delta([\nu^{-x-1} \rho, \nu^{-x} \rho]), \ldots, \delta([\nu^{-\alpha-1} \rho, \nu^{-\alpha} \rho]); \sigma_{sp}),
\end{gather*}
where $\sigma_{sp}$ is the unique irreducible subrepresentation of $\nu^{a} \rho \times \cdots \times \nu^{\alpha} \rho \rtimes \sigma$. 
\end{proposition}
\begin{proof}
Under the assumptions of the proposition, in $R(G)$ we have
\begin{gather*}
\delta([\nu^{a} \rho, \nu^{b} \rho]) \rtimes \delta(\rho, x; \sigma) = L(\delta([\nu^{-b} \rho, \nu^{-a} \rho]); \delta(\rho, x; \sigma)) + \\
L(\delta([\nu^{-x} \rho, \nu^{-a} \rho]); \delta(\rho, b; \sigma)).
\end{gather*}
Let us first determine the Aubert dual of $L(\delta([\nu^{-x} \rho, \nu^{-a} \rho]); \delta(\rho, b; \sigma))$. Using the third part of Proposition \ref{propulag} and Lemma \ref{lemaindprva}, we obtain that it is an irreducible subrepresentation of
\begin{equation*}
\nu^{-b} \rho \times \cdots \times \nu^{-x-1} \rho \rtimes \widehat{L(\delta([\nu^{-x} \rho, \nu^{-a} \rho]); \delta(\rho, x; \sigma))}.
\end{equation*}
Note that the induced representation $\nu^{x} \rho  \rtimes \delta(\rho, x; \sigma)$ is irreducible. Using the second part of Proposition \ref{propulag} we deduce that $L(\delta([\nu^{-x} \rho, \nu^{-a} \rho]); \delta(\rho, x; \sigma))$ is a subrepresentation of $\nu^{x} \rho \rtimes L(\delta([\nu^{-x+1} \rho, \nu^{-a} \rho]); \delta(\rho, x; \sigma))$, and then the third part of the same proposition gives an embedding
\begin{equation*}
L(\delta([\nu^{-x} \rho, \nu^{-a} \rho]); \delta(\rho, x; \sigma)) \hookrightarrow \nu^{x} \rho \times \nu^{x} \rho \rtimes L(\delta([\nu^{-x+1} \rho, \nu^{-a} \rho]); \delta(\rho, x-1; \sigma)).
\end{equation*}
We can continue in the same way to obtain the Aubert dual of $L(\delta([\nu^{-x} \rho, \nu^{-a} \rho])$; $\delta(\rho, x; \sigma))$ using Lemma \ref{lemainddruga}.

If $a = \alpha$, it follows that the Aubert dual of $L(\delta([\nu^{-x} \rho, \nu^{-a} \rho]); \delta(\rho, x; \sigma))$ is an irreducible subrepresentation of
\begin{equation*}
\nu^{-x} \rho \times \nu^{-x} \rho \times \cdots \times \nu^{-a} \rho \times \nu^{-a} \rho \rtimes \sigma.
\end{equation*}
If $a > \alpha$, it follows that the Aubert dual of $L(\delta([\nu^{-x} \rho, \nu^{-a} \rho]); \delta(\rho, x; \sigma))$ is an irreducible subrepresentation of
\begin{equation*}
\nu^{-x} \rho \times \nu^{-x} \rho \times \cdots \times \nu^{-a} \rho \times \nu^{-a} \rho \rtimes \widehat{\delta(\rho, a-1; \sigma)},
\end{equation*}
and it follows from \cite[Theorem~3.5]{Matic14} that $\widehat{\delta(\rho, a-1; \sigma)} \cong L(\nu^{-a+1} \rho, \ldots, \nu^{-\alpha} \rho; \sigma)$. Finally, if $a <\alpha$, it follows that the Aubert dual of $L(\delta([\nu^{-x} \rho, \nu^{-a} \rho]); \delta(\rho, x; \sigma))$ is an irreducible subrepresentation of
\begin{equation*}
\nu^{-x} \rho \times \nu^{-x} \rho \times \cdots \times \nu^{-\alpha} \rho \times \nu^{-\alpha} \rho \rtimes \widehat{L(\delta([\nu^{-\alpha+1} \rho, \nu^{-a} \rho]); \sigma)},
\end{equation*}
and the Aubert dual of $L(\delta([\nu^{-\alpha+1} \rho, \nu^{-a} \rho]); \sigma)$ is the unique irreducible subrepresentation of $\nu^{-\alpha + 1} \rho \times \cdots \times \nu^{-a} \rho \rtimes \sigma$, as before.

Let us now determine the Aubert dual of $L(\delta([\nu^{-b} \rho, \nu^{-a} \rho]); \delta(\rho, x; \sigma))$. First, using Lemma \ref{lemaindprva}, together with the first part of Proposition \ref{propulag}, we obtain that it is an irreducible subrepresentation of
\begin{equation*}
\nu^{-b} \rho \times \cdots \times \nu^{-x-2} \rho \rtimes \widehat{L(\delta([\nu^{-x-1} \rho, \nu^{-a} \rho]); \delta(\rho, x; \sigma))}.
\end{equation*}
Note that, by \cite[Proposition~3.1]{Mu3}, in $R(G)$ we have
\begin{gather*}
\delta([\nu^{a} \rho, \nu^{x+1} \rho]) \rtimes \delta(\rho, x; \sigma)) = L(\delta([\nu^{-x-1} \rho, \nu^{-a} \rho]); \delta(\rho, x; \sigma)) + \\
 L(\delta([\nu^{-x} \rho, \nu^{-a} \rho]); \delta(\rho, x+1; \sigma)).
\end{gather*}
Since $\delta([\nu^{a} \rho, \nu^{x} \rho]) \rtimes \delta(\rho, x; \sigma)$ is irreducible, the structural formula directly implies that $\nu^{x+1} \rho \otimes \delta([\nu^{a} \rho, \nu^{x} \rho]) \rtimes \delta(\rho, x; \sigma)$ is the unique irreducible constituent of the form $\nu^{x+1} \rho \otimes \pi$ appearing in $\mu^{\ast}(\delta([\nu^{a} \rho, \nu^{x+1} \rho]) \rtimes \delta(\rho, x; \sigma))$, which appears there with multiplicity one, and it obviously appears in \\$\mu^{\ast}(L(\delta([\nu^{-x} \rho, \nu^{-a} \rho]); \delta(\rho, x+1; \sigma)))$. Thus, $\mu^{\ast}(L(\delta([\nu^{-x-1} \rho, \nu^{-a} \rho]); \delta(\rho, x; \sigma)))$ does not contain an irreducible constituent of the form $\nu^{x+1} \rho \otimes \pi$.

Now, using the third part of Proposition \ref{propulag}, and then the first part of the same proposition, we obtain an embedding
\begin{gather*}
L(\delta([\nu^{-x-1} \rho, \nu^{-a} \rho]); \delta(\rho, x; \sigma)) \hookrightarrow \\
\zeta([\nu^{x} \rho, \nu^{x+1} \rho]) \rtimes L(\delta([\nu^{-x} \rho, \nu^{-a} \rho]); \delta(\rho, x-1; \sigma)).
\end{gather*}
Also, in the same way as before we conclude that $\mu^{\ast}(L(\delta([\nu^{-x} \rho, \nu^{-a} \rho]); \delta(\rho, x-1; \sigma)))$ does not contain an irreducible constituent of the form $\nu^{i} \rho \otimes \pi$, for $i \in \{ x, x+1 \}$. Using Lemma \ref{lemaindprva} and repeating this procedure, we obtain an embedding of the Aubert dual of $L(\delta([\nu^{-x-1} \rho, \nu^{-a} \rho]); \delta(\rho, x; \sigma))$.

If $a = \alpha$, it follows that the Aubert dual of $L(\delta([\nu^{-x-1} \rho, \nu^{-a} \rho]); \delta(\rho, x; \sigma))$ is an irreducible subrepresentation of
\begin{equation*}
\delta([\nu^{-x-1} \rho, \nu^{-x} \rho]) \times \cdots \times \delta([\nu^{-a-1} \rho, \nu^{-a} \rho]) \rtimes \widehat{L(\nu^{-\alpha} \rho; \sigma)},
\end{equation*}
and it follows from \cite[Theorem~3.5]{Matic14} that the Aubert dual of $L(\nu^{-\alpha} \rho; \sigma)$ is isomorphic to $\delta(\rho, \alpha; \sigma)$. Note that for $a = \alpha$ we have $\sigma_{sp} \cong \delta(\rho, \alpha; \sigma)$.

If $a > \alpha$, it follows that the Aubert dual of $L(\delta([\nu^{-x-1} \rho, \nu^{-a} \rho]); \delta(\rho, x; \sigma))$ is an irreducible subrepresentation of
\begin{equation*}
\delta([\nu^{-x-1} \rho, \nu^{-x} \rho]) \times \cdots \times \delta([\nu^{-a-1} \rho, \nu^{-a} \rho]) \rtimes \widehat{\delta(\rho, a-2; \sigma)},
\end{equation*}
and it follows from \cite[Theorem~3.5]{Matic14} that the Aubert dual of $\delta(\rho, a-2; \sigma)$ is the unique irreducible subrepresentation of $\nu^{-a+2} \rho \times \cdots \times \nu^{-\alpha} \rho \rtimes \sigma$.

If $a < \alpha$, it follows that the Aubert dual of $L(\delta([\nu^{-x-1} \rho, \nu^{-a} \rho]); \delta(\rho, x; \sigma))$ is an irreducible subrepresentation of
\begin{equation*}
\delta([\nu^{-x-1} \rho, \nu^{-x} \rho]) \times \cdots \times \delta([\nu^{-\alpha-1} \rho, \nu^{-\alpha} \rho]) \rtimes \widehat{L(\delta([\nu^{-\alpha} \rho, \nu^{-a} \rho ]); \sigma)},
\end{equation*}
and it follows from \cite[Theorem~3.5]{Matic14} that the Aubert dual of $L(\delta([\nu^{-\alpha} \rho, \nu^{-a} \rho ]); \sigma)$ is the unique irreducible subrepresentation of $\nu^{a} \rho \times \cdots \times \nu^{\alpha} \rho \rtimes \sigma$, which is strongly positive. This proves the proposition.
\end{proof}

Let us now consider the case $\rho_0 \not\cong \rho$. The following proposition can be proved in the same way as Proposition \ref{propvod}, using Lemma \ref{lemarazl}, details being left to the reader.

\begin{proposition}
Degenerate principal series $\zeta([\nu^{-b} \rho_0, \nu^{-a} \rho_0]) \rtimes \zeta(\rho, x; \sigma)$ is irreducible if and only if either $a > \beta$ or $b < \beta$. 
If $\zeta([\nu^{-b} \rho_0, \nu^{-a} \rho_0]) \rtimes \zeta(\rho, x; \sigma)$ reduces, in $R(G)$ we have
\begin{gather*}
\zeta([\nu^{-b} \rho_0, \nu^{-a} \rho_0]) \rtimes \zeta(\rho, x; \sigma) = \\ L(\nu^{-b} \rho_0, \ldots, \nu^{-a} \rho_0, \nu^{-x} \rho, \ldots, \nu^{-\alpha} \rho; \sigma) + \\
L(\nu^{-b} \rho_0, \ldots, \nu^{-\beta-1} \rho_0, \nu^{-x} \rho, \ldots, \nu^{-\alpha} \rho; \sigma_{sp}),
\end{gather*}
where $\sigma_{sp}$ is the unique irreducible subrepresentation of $\nu^{a} \rho_0 \times \cdots \times \nu^{\beta} \rho_0 \rtimes \sigma$.
\end{proposition}

\section{Case $a \leq 0$}\label{Section a0}

In this section we analyze the case when $a \leq 0$. To make the notation uniform, we let $\tau^{(1)} = \rho_0 \rtimes \sigma$ if $a \in \mathbb{Z}$ and $\tau^{(1)} = \sigma$ if $a \not\in \mathbb{Z}$. Also, if $a \not\in \mathbb{Z}$, let $\tau^{(2)}$ denote the unique irreducible (strongly positive) subrepresentation of $\nu^{\vh} \rho_0 \times \nu^{\frac{3}{2}} \rho_0 \times \cdots \times \nu^{\beta} \rho_0 \rtimes \sigma$. If $a \in \mathbb{Z}$, let $\tau'$ denote the unique irreducible (strongly positive) subrepresentation of $\nu \rho_0 \times \cdots \times \nu^{\beta} \rho_0 \rtimes \sigma$ and let $\tau^{(2)}$ denote an irreducible (tempered) subrepresentation of $\rho_0 \rtimes \tau'$ which does not contain an irreducible representation of the form $\nu \rho_0 \otimes \pi$ in the Jacquet module with respect to the appropriate parabolic subgroup. We note that such a subrepresentation of $\rho_0 \rtimes \tau'$ is unique by \cite[Section~4]{Tad6}.

For an irreducible self-contragredient cuspidal representation $\rho_1 \in R(GL)$ and an irreducible cuspidal representation $\sigma_1 \in R(G)$ such that $\nu^{\vh} \rho_1 \rtimes \sigma_1$ reduces, we denote by $\tau(\rho_1, \sigma_1)$ the unique irreducible tempered subrepresentation of $\delta([\nu^{- \frac{1}{2}} \rho_1, \nu^{\frac{1}{2}} \rho_1]) \rtimes \sigma_1$ which is not a subrepresentation of $\nu^{\frac{1}{2}} \rho_1 \times \nu^{\frac{1}{2}} \rho_1 \rtimes \sigma_1$,
Also, for a real number $y$ let $\lceil y\rceil$ stand for the smallest integer which is not smaller than $y$.

We will again first consider the more complicated case $\rho_0 \cong \rho$.
Let us first assume that $-a = b$.

\begin{proposition}
Degenerate principal series $\zeta([\nu^{-a} \rho, \nu^{a} \rho]) \rtimes \zeta(\rho, x; \sigma)$ is irreducible if and only if either $-a \leq \alpha - 2$ or $-a = x$. If $\alpha-2 < -a < x$, in $R(G)$ we have
\begin{gather*}
\zeta([\nu^{-a} \rho, \nu^{a} \rho]) \rtimes \zeta(\rho, x; \sigma) = \\
L(\nu^{-x} \rho, \ldots, \nu^{a-1} \rho, \nu^{a} \rho, \nu^{a} \rho, \nu^{a} \rho, \ldots, \nu^{-\alpha} \rho, \nu^{-\alpha} \rho, \nu^{-\alpha} \rho, \nu^{-\alpha+1} \rho, \nu^{-\alpha+1} \rho, \ldots, \\
\nu^{\lceil \alpha \rceil - \alpha - 1} \rho, \nu^{\lceil \alpha \rceil - \alpha - 1} \rho; \tau^{(1)}) + \pi,
\end{gather*}
where
\begin{gather*}
\pi \cong L(\nu^{-x} \rho, \ldots, \nu^{a-2} \rho, \delta([\nu^{a-1} \rho, \nu^{a} \rho]), \nu^{a} \rho, \ldots, \delta([\nu^{-\alpha} \rho, \nu^{-\alpha+1} \rho]), \nu^{-\alpha+1} \rho, \\
\nu^{-\alpha+2} \rho, \nu^{-\alpha+2} \rho, \ldots, \nu^{\lceil \alpha \rceil - \alpha - 1} \rho, \nu^{\lceil \alpha \rceil - \alpha - 1} \rho; \tau^{(1)}),
\end{gather*}
if $\alpha \geq \frac{3}{2}$,
\begin{equation*}
\pi \cong L(\nu^{-x} \rho, \ldots, \nu^{a-2} \rho, \delta([\nu^{a-1} \rho, \nu^{a} \rho]), \nu^{a} \rho, \ldots, \delta([\nu^{-2} \rho, \nu^{-1} \rho]), \nu^{-1} \rho, \delta([\nu^{-1} \rho, \rho]); \sigma),
\end{equation*}
if $\alpha = 1$,
\begin{gather*}
\pi \cong L(\nu^{-x} \rho, \ldots, \nu^{a-2} \rho, \delta([\nu^{a-1} \rho, \nu^{a} \rho]), \nu^{a} \rho, \ldots, \delta([\nu^{-\frac{5}{2}} \rho, \nu^{-\frac{3}{2}} \rho]), \nu^{-\frac{3}{2}} \rho, \\
\delta([\nu^{-\frac{3}{2}} \rho, \nu^{-\frac{1}{2}} \rho]); \tau(\rho, \sigma)),
\end{gather*}
if $\alpha = \frac{1}{2}$.

If $-a > x$, in $R(G)$ we have
\begin{gather*}
\zeta([\nu^{-a} \rho, \nu^{a} \rho]) \rtimes \zeta(\rho, x; \sigma) = \\
L(\nu^{a} \rho, \nu^{a} \rho, \ldots, \nu^{-x-1} \rho, \nu^{-x-1} \rho, \nu^{-x} \rho, \nu^{-x} \rho, \nu^{-x} \rho, \ldots, \nu^{-\alpha} \rho, \nu^{-\alpha} \rho, \nu^{-\alpha} \rho, \\
\nu^{-\alpha+1} \rho, \nu^{-\alpha+1} \rho, \ldots, \nu^{\lceil \alpha \rceil - \alpha - 1} \rho, \nu^{\lceil \alpha \rceil - \alpha - 1} \rho; \tau^{(1)}) + \\
L(\nu^{a} \rho, \nu^{a} \rho, \ldots, \nu^{-x-2} \rho, \nu^{-x-2} \rho, \nu^{-x-1} \rho, \delta([\nu^{-x-1} \rho, \nu^{-x} \rho]), \ldots, \\
\nu^{-\alpha-1} \rho, \delta([\nu^{-\alpha-1} \rho, \nu^{-\alpha} \rho]),
\nu^{-\alpha} \rho, \ldots, \nu^{\lceil \alpha \rceil - \alpha - 1} \rho; \tau^{(2)}).
\end{gather*}
\end{proposition}
\begin{proof}
Reducibility of $\delta([\nu^{-a} \rho, \nu^{a} \rho]) \rtimes \delta(\rho, x; \sigma)$ is an integral part of the classification of discrete series. If such an induced representation reduces, it is a direct sum of two mutually non-isomorphic irreducible tempered representation, whose Aubert duals can be easily obtained from \cite[Theorem~4.11, Theorem~4.16, Theorem~4.21]{Matic10}.
\end{proof}

Now we deal with the case $-a < b$. The reducibility criterion follows from \cite[Theorem~4.1(i)]{Mu3}:

\begin{proposition}
Degenerate principal series $\zeta([\nu^{-b} \rho, \nu^{-a} \rho]) \rtimes \zeta(\rho, x; \sigma)$ is irreducible if and only if one of the following holds:
\begin{itemize}
\item $b < \alpha-1$,
\item $-a < \alpha-1$ and $b = x$.
\end{itemize}
\end{proposition}

Other possibilities will be studied using a case-by-case consideration.

\begin{proposition} \label{propmanjeprva}
If $-a > x$, in $R(G)$ we have
\begin{gather*}
\zeta([\nu^{-b} \rho, \nu^{-a} \rho]) \rtimes \zeta(\rho, x; \sigma) = \\
L(\nu^{-b} \rho, \ldots, \nu^{a-1} \rho, \nu^{a} \rho, \nu^{a} \rho, \ldots, \nu^{-x-1} \rho, \nu^{-x-1} \rho, \nu^{-x} \rho, \nu^{-x} \rho, \nu^{-x} \rho, \ldots, \\
\nu^{-\alpha} \rho, \nu^{-\alpha} \rho, \nu^{-\alpha} \rho, \nu^{-\alpha+1} \rho, \nu^{-\alpha+1} \rho, \ldots, \nu^{\lceil \alpha \rceil - \alpha - 1} \rho, \nu^{\lceil \alpha \rceil - \alpha - 1} \rho; \tau^{(1)}) + \\
L(\nu^{-b} \rho, \ldots, \nu^{a-1} \rho, \nu^{a} \rho, \nu^{a} \rho, \ldots, \nu^{-x-2} \rho, \nu^{-x-2} \rho, \nu^{-x-1} \rho, \delta([\nu^{-x-1} \rho, \nu^{-x} \rho]), \ldots, \\
\nu^{-\alpha-1} \rho, \delta([\nu^{-\alpha-1} \rho, \nu^{-\alpha} \rho]), \nu^{-\alpha} \rho, \ldots, \nu^{\lceil \alpha \rceil - \alpha - 1} \rho; \tau^{(2)}) + \\
L(\nu^{-b} \rho, \ldots, \nu^{a-2} \rho, \delta([\nu^{a-1} \rho, \nu^{a} \rho]), \ldots, \delta([\nu^{-x-2} \rho, \nu^{-x-1} \rho]), \\
\delta([\nu^{-x-1} \rho, \nu^{-x} \rho]), \nu^{-x} \rho, \ldots, \delta([\nu^{-\alpha-1} \rho, \nu^{-\alpha} \rho]), \nu^{-\alpha} \rho, \nu^{-\alpha+1} \rho, \ldots,
\nu^{\lceil \alpha \rceil - \alpha - 1} \rho; \tau^{(2)}).
\end{gather*}
\end{proposition}
\begin{proof}
By \cite[Theorem~2.1]{Mu3}, in $R(G)$ we have
\begin{gather*}
\delta([\nu^{a} \rho, \nu^{b} \rho]) \rtimes \delta(\rho, x; \sigma) = L(\delta([\nu^{-b} \rho, \nu^{-a} \rho]); \delta(\rho, x; \sigma)) + \sigma_1 + \sigma_2,
\end{gather*}
where $\sigma_1, \sigma_2$ are mutually non-isomorphic discrete series representations. Aubert duals of $\sigma_1$ and $\sigma_2$ have been obtained in \cite[Theorems~4.11,~4.16]{Matic10}. It remains to determine the Aubert dual of $L(\delta([\nu^{-b} \rho, \nu^{-a} \rho]); \delta(\rho, x; \sigma))$.

Using Proposition \ref{propulag}$(1)$ and Lemma \ref{lemaindprva}, we deduce that the Aubert dual of $L(\delta([\nu^{-b} \rho, \nu^{-a} \rho]); \delta(\rho, x; \sigma))$ is a subrepresentation of $\nu^{-b} \rho \times \cdots \times \nu^{a-2} \rho \rtimes \widehat{L(\delta([\nu^{a-1} \rho, \nu^{-a} \rho]); \delta(\rho, x; \sigma))}$. Now Proposition \ref{propulag}$(2)$ and $(1)$ imply that $L(\delta([\nu^{a-1} \rho, \nu^{-a} \rho]); \delta(\rho, x; \sigma))$ is a subrepresentation of $\nu^{-a} \rho \times \nu^{-a+1} \rho \rtimes L(\delta([\nu^{a} \rho, \nu^{-a-1} \rho]); \delta(\rho, x; \sigma))$. It can be seen, in the same way as in the proof of Proposition \ref{propnulaprva}, that $\mu^{\ast}(L(\delta([\nu^{a-1} \rho, \nu^{-a} \rho]); \delta(\rho, x; \sigma)))$ does not contain an irreducible constituent of the form $\nu^{-a+1} \rho \otimes \pi$, so
$L(\delta([\nu^{a-1} \rho, \nu^{-a} \rho]); \delta(\rho, x; \sigma))$ is a subrepresentation of $\zeta([\nu^{-a} \rho, \nu^{-a+1} \rho]) \rtimes L(\delta([\nu^{a} \rho, \nu^{-a-1} \rho]); \delta(\rho, x; \sigma))$.

Using Lemma \ref{lemaindtreca} and repeating this procedure, we obtain that the Aubert dual of $L(\delta([\nu^{a-1} \rho, \nu^{-a} \rho]); \delta(\rho, x; \sigma))$ is a subrepresentation of
\begin{equation*}
\delta([\nu^{a-1} \rho, \nu^{a} \rho]) \times \cdots \times \delta([\nu^{-x-2} \rho, \nu^{-x-1} \rho]) \rtimes \widehat{L(\delta([\nu^{-x-1} \rho, \nu^{x} \rho]); \delta(\rho, x; \sigma))}.
\end{equation*}

We now determine the Aubert dual of $L(\delta([\nu^{-x-1} \rho, \nu^{x} \rho]); \delta(\rho, x; \sigma))$. We first show the following embedding:
\begin{gather*}
L(\delta([\nu^{-x-1} \rho, \nu^{x} \rho]); \delta(\rho, x; \sigma)) \hookrightarrow \nonumber \\
\zeta([\nu^{x} \rho, \nu^{x+1} \rho]) \times \nu^{x} \rho \rtimes L(\delta([\nu^{-x} \rho, \nu^{x-1} \rho]); \delta(\rho, x-1; \sigma)). \end{gather*}

Note that
\begin{align*}
L(\delta([\nu^{-x-1} \rho, \nu^{x} \rho]); \delta(\rho, x; \sigma)) & \hookrightarrow \delta([\nu^{-x-1} \rho, \nu^{x} \rho]) \rtimes \delta(\rho, x; \sigma) \\
& \hookrightarrow \delta([\nu^{-x-1} \rho, \nu^{x} \rho]) \times \nu^{x} \rho \rtimes \delta(\rho, x-1; \sigma) \\
& \cong \nu^{x} \rho \times \delta([\nu^{-x-1} \rho, \nu^{x} \rho])\rtimes \delta(\rho, x-1; \sigma)\\
& \hookrightarrow \nu^{x} \rho \times \nu^{x} \rho \times \delta([\nu^{-x-1} \rho, \nu^{x-1} \rho])\rtimes \delta(\rho, x-1; \sigma).
\end{align*}
Consequently, there is an irreducible subquotient $\pi$ of $\delta([\nu^{-x-1} \rho, \nu^{x-1} \rho])\rtimes \delta(\rho, x-1; \sigma)$ such that $L(\delta([\nu^{-x-1} \rho, \nu^{x} \rho]); \delta(\rho, x; \sigma))$ is a subrepresentation of $\nu^{x} \rho \times \nu^{x} \rho \rtimes \pi$. Since $\mu^{\ast}(L(\delta([\nu^{-x-1} \rho, \nu^{x} \rho]); \delta(\rho, x; \sigma))) \geq \delta([\nu^{-x-1} \rho, \nu^{x} \rho]) \otimes \delta(\rho, x; \sigma)$, it follows at once that $\pi \cong L(\delta([\nu^{-x-1} \rho, \nu^{x-1} \rho]); \delta(\rho, x-1; \sigma))$.

By Proposition \ref{propulag}$(1)$, $L(\delta([\nu^{-x-1} \rho, \nu^{x-1} \rho]); \delta(\rho, x-1; \sigma))$ is a subrepresentation of $\nu^{x+1} \rho \rtimes L(\delta([\nu^{-x} \rho, \nu^{x-1} \rho]); \delta(\rho, x-1; \sigma))$, so there is an irreducible subquotient $\pi_1$ of $\nu^{x} \rho \times \nu^{x} \rho \times \nu^{x+1} \rho$ such that
$L(\delta([\nu^{-x-1} \rho, \nu^{x} \rho]); \delta(\rho, x; \sigma))$ is a subrepresentation of $\pi_1 \rtimes L(\delta([\nu^{-x} \rho, \nu^{x-1} \rho]); \delta(\rho, x-1; \sigma))$.

From \cite[Theorem~4.1(ii)]{Mu3} follows that in $R(G)$ we have
\begin{equation*}
\delta([\nu^{-x} \rho, \nu^{x+1} \rho]) \rtimes \delta(\rho, x; \sigma) = L(\delta([\nu^{-x-1} \rho, \nu^{x} \rho]); \delta(\rho, x; \sigma)) + \tau_{temp},
\end{equation*}
where $\tau_{temp}$ is an irreducible tempered subrepresentation of $\delta([\nu^{-x} \rho, \nu^{x} \rho]) \rtimes \delta(\rho, x+1; \sigma)$, which is also a subrepresentation of $\delta([\nu^{-x} \rho, \nu^{x+1} \rho]) \rtimes \delta(\rho, x; \sigma)$. Thus, $$\mu^{\ast}(\tau_{temp}) \geq \nu^{x+1} \rho \otimes \delta([\nu^{-x} \rho, \nu^{x} \rho]) \rtimes \delta(\rho, x; \sigma).$$
It follows from the structural formula and irreducibility of $\delta([\nu^{-x} \rho, \nu^{x} \rho]) \rtimes \delta(\rho, x; \sigma)$ that
$\nu^{x+1} \rho \otimes \delta([\nu^{-x} \rho, \nu^{x} \rho]) \rtimes \delta(\rho, x; \sigma)$ is the unique irreducible constituent of $\mu^{\ast}(\delta([\nu^{-x} \rho, \nu^{x+1} \rho]) \rtimes \delta(\rho, x; \sigma))$ of the form $\nu^{x+1} \rho \otimes \pi$, so $\mu^{\ast}(L(\delta([\nu^{-x-1} \rho, \nu^{x} \rho]); \delta(\rho, x; \sigma)))$ does not contain an irreducible constituent of such a form. Consequently, $\pi_1 \cong \zeta([\nu^{x} \rho, \nu^{x+1} \rho]) \times \nu^{x} \rho $.

In the same way it can be seen that $\mu^{\ast}(L(\delta([\nu^{-x} \rho, \nu^{x-1} \rho]); \delta(\rho, x-1; \sigma)))$ does not contain irreducible constituents of the form $\nu^{y} \rho \otimes \pi$, for $y \in \{ x, x+1 \}$.

Using Lemma \ref{lemainddruga}, we obtain that $\widehat{L(\delta([\nu^{-x-1} \rho, \nu^{x} \rho]); \delta(\rho, x; \sigma))}$ is a subrepresentation of $\delta([\nu^{-x-1} \rho,$ $\nu^{-x} \rho]) \times \nu^{-x} \rho \rtimes \widehat{L(\delta([\nu^{-x} \rho, \nu^{x-1} \rho]); \delta(\rho, x-1; \sigma)}$. 

Repeating this procedure until $x=\alpha$, we also obtain that the Aubert dual of $L(\delta([\nu^{-x-1} \rho,$ $\nu^{x} \rho]); \delta(\rho, x; \sigma))$ is a subrepresentation of
\begin{equation*}
\delta([\nu^{-x-1} \rho, \nu^{-x} \rho]) \times \nu^{-x} \rho \times \cdots \times \delta([\nu^{-\alpha-1} \rho, \nu^{-\alpha} \rho]) \times \nu^{-\alpha} \rho  \rtimes \widehat{L(\delta([\nu^{-\alpha} \rho, \nu^{\alpha-1} \rho]); \sigma)}.
\end{equation*}
It follows from \cite[Lemma~4.10]{Matic10} that the Aubert dual of $L(\delta([\nu^{-\alpha} \rho, \nu^{\alpha-1} \rho]); \sigma)$ is the unique irreducible subrepresentation of $\nu^{-\alpha+1} \rho \times \cdots \times \nu^{\lceil \alpha \rceil - \alpha - 1} \rho \rtimes \tau^{(2)}$, and the proposition is proved.
\end{proof}

\begin{proposition}  \label{propmanjedruga}
Suppose that $\alpha-1 \leq -a < b < x$. Let
\begin{gather*}
\pi_1 \cong L(\nu^{-x} \rho, \ldots, \nu^{-b-2} \rho, \delta([\nu^{-b-1} \rho, \nu^{-b} \rho]), \ldots, \delta([\nu^{a-2} \rho, \nu^{a-1} \rho]), \\
\delta([\nu^{a-1} \rho, \nu^{a} \rho]), \nu^{a} \rho, \ldots, \delta([\nu^{-\alpha} \rho, \nu^{-\alpha+1} \rho]), \nu^{-\alpha+1} \rho, \\
\nu^{-\alpha+2} \rho, \nu^{-\alpha+2} \rho, \ldots, \nu^{\lceil \alpha \rceil - \alpha - 1} \rho, \nu^{\lceil \alpha \rceil - \alpha - 1} \rho; \tau^{(1)})
\end{gather*}
if $\alpha \geq \frac{3}{2}$,
\begin{gather*}
\pi_1 \cong L(\nu^{-x} \rho, \ldots, \nu^{-b-2} \rho, \delta([\nu^{-b-1} \rho, \nu^{-b} \rho]), \ldots, \delta([\nu^{a-2} \rho, \nu^{a-1} \rho]), \\
\delta([\nu^{a-1} \rho, \nu^{a} \rho]), \nu^{a} \rho, \ldots, \delta([\nu^{-2} \rho, \nu^{-1} \rho]), \nu^{-1} \rho, \delta([\nu^{-1} \rho, \rho]); \sigma)
\end{gather*}
if $\alpha = 1$, and
\begin{gather*}
\pi_1 \cong L(\nu^{-x} \rho, \ldots, \nu^{-b-2} \rho, \delta([\nu^{-b-1} \rho, \nu^{-b} \rho]), \ldots, \delta([\nu^{a-2} \rho, \nu^{a-1} \rho]), \\
\delta([\nu^{a-1} \rho, \nu^{a} \rho]), \nu^{a} \rho, \ldots, \delta([\nu^{-\frac{5}{2}} \rho, \nu^{-\frac{3}{2}} \rho]), \nu^{-\frac{3}{2}} \rho,
\delta([\nu^{-\frac{3}{2}} \rho, \nu^{-\frac{1}{2}} \rho]); \tau(\rho, \sigma)),
\end{gather*}
if $\alpha = \frac{1}{2}$.

Also, let
\begin{gather*}
\pi_2 \cong L(\nu^{-x} \rho, \ldots, \nu^{-b-1} \rho, \nu^{-b} \rho, \nu^{-b} \rho, \ldots, \nu^{a-2} \rho, \nu^{a-2} \rho,
\nu^{a-1} \rho, \delta([\nu^{a-1} \rho, \nu^{a} \rho]), \ldots, \\
\nu^{-\alpha} \rho, \delta([\nu^{-\alpha} \rho, \nu^{-\alpha+1} \rho]), \nu^{-\alpha+1} \rho, \nu^{-\alpha+2} \rho, \nu^{-\alpha+2} \rho, \ldots,
\nu^{\lceil \alpha \rceil - \alpha - 1} \rho, \nu^{\lceil \alpha \rceil - \alpha - 1} \rho; \tau^{(1)}).
\end{gather*}
if $\alpha \geq \frac{3}{2}$,
\begin{gather*}
\pi_2 \cong L(\nu^{-x} \rho, \ldots, \nu^{-b-1} \rho, \nu^{-b} \rho, \nu^{-b} \rho, \ldots, \nu^{a-2} \rho, \nu^{a-2} \rho,
\nu^{a-1} \rho, \delta([\nu^{a-1} \rho, \nu^{a} \rho]), \ldots, \\
\nu^{-1} \rho, \delta([\nu^{-1} \rho, \rho]), \sigma),
\end{gather*}
if $\alpha = 1$, and
\begin{gather*}
\pi_2 \cong L(\nu^{-x} \rho, \ldots, \nu^{-b-1} \rho, \nu^{-b} \rho, \nu^{-b} \rho, \ldots, \nu^{a-2} \rho, \nu^{a-2} \rho,
\nu^{a-1} \rho, \delta([\nu^{a-1} \rho, \nu^{a} \rho]), \ldots, \\
\nu^{-\frac{3}{2}} \rho, \delta([\nu^{-\frac{3}{2}} \rho, \nu^{-\frac{1}{2}} \rho]); \tau(\rho, \sigma)),
\end{gather*}
if $\alpha = \vh$.

Then in $R(G)$ we have
\begin{gather*}
\zeta([\nu^{-b} \rho, \nu^{-a} \rho]) \rtimes \zeta(\rho, x; \sigma) = \\
L(\nu^{-x} \rho, \ldots, \nu^{-b-1} \rho, \nu^{-b} \rho, \nu^{-b} \rho, \ldots, \nu^{-a-1} \rho, \nu^{-a-1} \rho, \nu^{-a} \rho, \nu^{-a} \rho, \nu^{-a} \rho, \ldots, \\
\nu^{-\alpha} \rho, \nu^{-\alpha} \rho, \nu^{-\alpha} \rho, \nu^{-\alpha+1} \rho, \nu^{-\alpha+1} \rho, \ldots, \nu^{\lceil \alpha \rceil - \alpha - 1} \rho, \nu^{\lceil \alpha \rceil - \alpha - 1} \rho; \tau^{(1)}) + \pi_1 + \pi_2.
\end{gather*}
\end{proposition}
\begin{proof}
Again, by \cite[Theorem~2.1]{Mu3}, in $R(G)$ we have
\begin{gather*}
\delta([\nu^{a} \rho, \nu^{b} \rho]) \rtimes \delta(\rho, x; \sigma) = L(\delta([\nu^{-b} \rho, \nu^{-a} \rho]); \delta(\rho, x; \sigma)) + \sigma_1 + \sigma_2,
\end{gather*}
where $\sigma_1, \sigma_2$ are mutually non-isomorphic discrete series representations.

Similarly as in the previous proposition, it is enough to determine the Aubert dual of $L(\delta([\nu^{-b} \rho, \nu^{-a} \rho]); \delta(\rho, x; \sigma))$. Using Proposition \ref{propulag}$(3)$ and Lemma \ref{lemaindprva}, we deduce that the Aubert dual of $L(\delta([\nu^{-b} \rho, \nu^{-a} \rho]); \delta(\rho, x; \sigma))$ is an irreducible subrepresentation of
\begin{equation*}
\nu^{-x} \rho \times \cdots \times \nu^{-b-1} \rho \rtimes \widehat{L(\delta([\nu^{-b} \rho, \nu^{-a} \rho]); \delta(\rho, b; \sigma))}.
\end{equation*}
If $b > -a+1$, we have the following embeddings and isomorphisms:
\begin{align*}
L(\delta([\nu^{-b} \rho, \nu^{-a} \rho]); \delta(\rho, b; \sigma)) & \hookrightarrow \delta([\nu^{-b+1} \rho, \nu^{-a} \rho]) \times \nu^{-b} \rho \rtimes \delta(\rho, b; \sigma) \\
& \cong \delta([\nu^{-b+1} \rho, \nu^{-a} \rho]) \times \nu^{b} \rho \rtimes \delta(\rho, b; \sigma) \\
& \cong \nu^{b} \rho \times \delta([\nu^{-b+1} \rho, \nu^{-a} \rho]) \rtimes \delta(\rho, b; \sigma) \\
& \hookrightarrow \nu^{b} \rho \times \delta([\nu^{-b+1} \rho, \nu^{-a} \rho]) \times \nu^{b} \rho  \rtimes \delta(\rho, b-1; \sigma) \\
& \cong \nu^{b} \rho \times \nu^{b} \rho \times \delta([\nu^{-b+1} \rho, \nu^{-a} \rho]) \rtimes \delta(\rho, b-1; \sigma).
\end{align*}
Thus, there is an irreducible subquotient $\pi$ of $\delta([\nu^{-b+1} \rho, \nu^{-a} \rho]) \rtimes \delta(\rho, b-1; \sigma)$ such that $L(\delta([\nu^{-b} \rho, \nu^{-a} \rho]); \delta(\rho, b; \sigma))$ is a subrepresentation of $\nu^{b} \rho \times \nu^{b} \rho \rtimes \pi$. Since $\mu^{\ast}(L(\delta([\nu^{-b} \rho, \nu^{-a} \rho]); \delta(\rho, b; \sigma))) \geq \delta([\nu^{-b} \rho, \nu^{-a} \rho]) \otimes \delta(\rho, b; \sigma)$, it follows that $\pi \cong L(\delta([\nu^{-b+1} \rho, \nu^{-a} \rho]); \delta(\rho, b-1; \sigma))$. Obviously, $\mu^{\ast}(L(\delta([\nu^{-b+1} \rho, \nu^{-a} \rho])$; $\delta(\rho, b-1; \sigma)))$ does not contain an irreducible constituent of the form $\nu^{b} \rho \otimes \pi_1$. Repeated application of this procedure and Lemma \ref{lemainddruga} lead us to an embedding
\begin{gather*}
\widehat{L(\delta([\nu^{-b} \rho, \nu^{-a} \rho]); \delta(\rho, b; \sigma))} \hookrightarrow \\
\nu^{-b} \rho \times \nu^{-b} \rho \times \cdots \times \nu^{a-2} \rho \times \nu^{a-2} \rho \rtimes \widehat{L(\delta([\nu^{a-1} \rho, \nu^{-a} \rho]); \delta(\rho, -a+1; \sigma))}.
\end{gather*}
Thus, it remains to determine $\widehat{L(\delta([\nu^{a-1} \rho, \nu^{-a} \rho]); \delta(\rho, -a+1; \sigma))}$. Proposition \ref{propulag}$(2)$ implies that $L(\delta([\nu^{a-1} \rho, \nu^{-a} \rho]); \delta(\rho, -a+1; \sigma))$ is a subrepresentation of $\nu^{-a} \rho \rtimes L(\delta([\nu^{a-1} \rho, \nu^{-a-1} \rho]); \delta(\rho, -a+1; \sigma))$, and in the same way as before we get
\begin{gather*}
L(\delta([\nu^{a-1} \rho, \nu^{-a} \rho]); \delta(\rho, -a+1; \sigma)) \hookrightarrow \\
\nu^{-a} \rho \times \nu^{-a+1} \rho \times \nu^{-a+1} \rho \rtimes L(\delta([\nu^{a} \rho, \nu^{-a-1} \rho]); \delta(\rho, -a; \sigma)).
\end{gather*}
By \cite[Theorem~4.1]{Mu3}, in $R(G)$ we have
\begin{gather*}
\delta([\nu^{a} \rho, \nu^{-a+1} \rho]) \rtimes \delta(\rho, -a+1; \sigma) = \\
L(\delta([\nu^{a-1} \rho, \nu^{-a} \rho]); \delta(\rho, -a+1; \sigma)) + \tau_{temp},
\end{gather*}
where $\tau_{temp}$ is the unique common irreducible subrepresentation of $$\delta([\nu^{a-1} \rho, \nu^{-a+1} \rho]) \rtimes \delta(\rho, -a; \sigma)$$
and 
$$\delta([\nu^{a} \rho, \nu^{-a+1} \rho]) \rtimes \delta(\rho, -a+1; \sigma).$$
From the structural formula we obtain that 
$$\nu^{-a+1} \rho \times \nu^{-a+1} \rho \otimes \delta([\nu^{a} \rho, \nu^{-a} \rho]) \rtimes \delta(\rho, -a; \sigma)$$
is the unique irreducible constituent of $\mu^{\ast}(\delta([\nu^{a} \rho, \nu^{-a+1} \rho]) \rtimes \delta(\rho, -a+1; \sigma))$ of the form $\nu^{-a+1} \rho \times \nu^{-a+1} \rho \otimes \pi'$, which appears there with multiplicity one, and by Frobenius reciprocity it is contained in $\mu^{\ast}(\tau_{temp})$. Thus, $$\mu^{\ast}(L(\delta([\nu^{a-1} \rho, \nu^{-a} \rho]); \delta(\rho, -a+1; \sigma)))$$
does not contain an irreducible constituent of the form $\nu^{-a+1} \rho \times \nu^{-a+1} \rho \otimes \pi'$, which yields
\begin{gather*}
L(\delta([\nu^{a-1} \rho, \nu^{-a} \rho]); \delta(\rho, -a+1; \sigma)) \hookrightarrow \\
\zeta([\nu^{-a} \rho, \nu^{-a+1} \rho]) \times \nu^{-a+1} \rho \rtimes L(\delta([\nu^{a} \rho, \nu^{-a-1} \rho]); \delta(\rho, -a; \sigma)).
\end{gather*}
Also, $\mu^{\ast}(L(\delta([\nu^{a} \rho, \nu^{-a-1} \rho]); \delta(\rho, -a; \sigma)))$ does not contain an irreducible constituent of the form $\nu^{-a+1} \rho \otimes \pi'_1$, so using Lemma \ref{lemaindtreca} and a repeated application of this procedure, we get that the Aubert dual of $L(\delta([\nu^{a-1} \rho, \nu^{-a} \rho]);$ $\delta(\rho, -a+1; \sigma))$ is an irreducible subrepresentation of
\begin{gather*}
\nu^{a-1} \rho \times \delta([\nu^{a-1} \rho, \nu^{a} \rho]) \times \cdots \times \nu^{-\alpha-1} \rho \times \delta([\nu^{-\alpha-1} \rho, \nu^{-\alpha} \rho]) \rtimes \\ \widehat{L(\delta([\nu^{-\alpha} \rho, \nu^{\alpha-1} \rho]); \delta(\rho, \alpha; \sigma))}.
\end{gather*}
If $\alpha = \vh$, by \cite[Lemma~4.10]{Matic10} we have $\widehat{L(\delta([\nu^{-\alpha} \rho, \nu^{\alpha-1} \rho]); \delta(\rho, \alpha; \sigma))} \cong \tau(\rho, \sigma)$. If $\alpha > \vh$, in the same way as before we get
\begin{gather*}
\widehat{L(\delta([\nu^{-\alpha} \rho, \nu^{\alpha-1} \rho]); \delta(\rho, \alpha; \sigma))} \hookrightarrow \\
\nu^{-\alpha} \rho \times \delta([\nu^{-\alpha} \rho, \nu^{-\alpha+1} \rho]) \rtimes \widehat{L(\delta([\nu^{-\alpha+1} \rho, \nu^{\alpha-2} \rho]); \sigma)}.
\end{gather*}
For $\alpha = 1$, we have $L(\delta([\nu^{-\alpha+1} \rho, \nu^{\alpha-2} \rho]); \sigma) \cong \sigma$, and for $\alpha \geq \frac{3}{2}$ we have
\begin{gather*}
\widehat{L(\delta([\nu^{-\alpha+1} \rho, \nu^{\alpha-2} \rho]); \sigma)} \hookrightarrow \\
\nu^{-\alpha+1} \rho \times \nu^{-\alpha+2} \rho \times \nu^{-\alpha+2} \rho \times \cdots \times
\nu^{\lceil \alpha \rceil - \alpha - 1} \rho \times \nu^{\lceil \alpha \rceil - \alpha - 1} \rho \rtimes \tau^{(1)}.
\end{gather*}
This ends the proof.
\end{proof}

\begin{proposition}
Suppose that $\alpha-1 \leq -a < x$ and $x < b$. Let
\begin{gather*}
\pi_1 \cong
L(\nu^{-b} \rho, \ldots, \nu^{-x-1} \rho, \nu^{-x} \rho, \nu^{-x} \rho, \ldots, \nu^{a-2} \rho, \nu^{a-2} \rho, \nu^{a-1} \rho, \delta([\nu^{a-1} \rho, \nu^{a} \rho]), \\
\ldots, \nu^{-\frac{3}{2}} \rho, \delta([\nu^{-\frac{3}{2}} \rho, \nu^{-\vh} \rho]); \tau(\rho, \sigma)),
\end{gather*}
if $\alpha = \vh$,
\begin{gather*}
\pi_1 \cong
L(\nu^{-b} \rho, \ldots, \nu^{-x-1} \rho, \nu^{-x} \rho, \nu^{-x} \rho, \ldots, \nu^{a-2} \rho, \nu^{a-2} \rho, \nu^{a-1} \rho, \delta([\nu^{a-1} \rho, \nu^{a} \rho]), \\
\ldots, \nu^{-1} \rho, \delta([\nu^{-1} \rho, \rho]); \sigma),
\end{gather*}
if $\alpha = 1$, and
\begin{gather*}
\pi_1 \cong
L(\nu^{-b} \rho, \ldots, \nu^{-x-1} \rho, \nu^{-x} \rho, \nu^{-x} \rho, \ldots, \nu^{a-2} \rho, \nu^{a-2} \rho, \nu^{a-1} \rho, \delta([\nu^{a-1} \rho, \nu^{a} \rho]), \ldots, \\
\nu^{-\alpha} \rho, \delta([\nu^{-\alpha} \rho, \nu^{-\alpha+ 1} \rho]), \nu^{-\alpha+1} \rho, \nu^{-\alpha+2} \rho, \nu^{-\alpha+2} \rho, \ldots, \nu^{\lceil \alpha \rceil - \alpha - 1} \rho, \nu^{\lceil \alpha \rceil - \alpha - 1} \rho; \tau^{(1)}),
\end{gather*}
if $\alpha \geq \frac{3}{2}$.

Let
\begin{gather*}
\pi_2 \cong
L(\nu^{-b} \rho, \ldots, \nu^{-x-2} \rho, \delta([\nu^{-x-1} \rho, \nu^{-x} \rho]), \ldots, \delta([\nu^{a-3} \rho, \nu^{a-2} \rho]), \\ \delta([\nu^{a-2} \rho, \nu^{a} \rho]), \ldots, \delta([\nu^{-\frac{3}{2}} \rho, \nu^{\vh} \rho]); \sigma),
\end{gather*}
if $\alpha = \vh$,
\begin{gather*}
\pi_2 \cong
L(\nu^{-b} \rho, \ldots, \nu^{-x-2} \rho, \delta([\nu^{-x-1} \rho, \nu^{-x} \rho]), \ldots, \delta([\nu^{a-3} \rho, \nu^{a-2} \rho]), \\ \delta([\nu^{a-2} \rho, \nu^{a} \rho]), \ldots, \delta([\nu^{-2} \rho, \rho]); \delta(\rho, 1; \sigma)),
\end{gather*}
if $\alpha = 1$, and
\begin{gather*}
\pi_2 \cong
L(\nu^{-b} \rho, \ldots, \nu^{-x-2} \rho, \delta([\nu^{-x-1} \rho, \nu^{-x} \rho]), \ldots, \delta([\nu^{a-3} \rho, \nu^{a-2} \rho]), \\ \delta([\nu^{a-2} \rho, \nu^{a} \rho]), \ldots, \delta([\nu^{-\alpha-1} \rho, \nu^{-\alpha+1} \rho]), \nu^{-\alpha+2} \rho, \ldots, \nu^{\lceil \alpha \rceil - \alpha - 1} \rho; \tau^{(2)}),
\end{gather*}
if $\alpha \geq \frac{3}{2}$.

Then in $R(G)$ we have
\begin{gather*}
\zeta([\nu^{-b} \rho, \nu^{-a} \rho]) \rtimes \zeta(\rho, x; \sigma) = \\
L(\nu^{-b} \rho, \ldots, \nu^{-x-1} \rho, \nu^{-x} \rho, \nu^{-x} \rho, \ldots, \nu^{a-1} \rho, \nu^{a-1} \rho, \nu^{a} \rho, \nu^{a} \rho, \nu^{a} \rho, \ldots, \\
\nu^{-\alpha} \rho, \nu^{-\alpha} \rho, \nu^{-\alpha} \rho, \nu^{-\alpha+1} \rho, \nu^{-\alpha+1} \rho, \ldots, \nu^{\lceil \alpha \rceil - \alpha - 1} \rho, \nu^{\lceil \alpha \rceil - \alpha - 1} \rho; \tau^{(1)}) + \\
L(\nu^{-b} \rho, \ldots, \nu^{-x-2} \rho, \delta([\nu^{-x-1} \rho, \nu^{-x} \rho]), \ldots, \delta([\nu^{a-2} \rho, \nu^{a-1} \rho]),  \delta([\nu^{a-1} \rho, \nu^{a} \rho]), \nu^{a} \rho, \\
\ldots, \delta([\nu^{-\alpha-1} \rho, \nu^{-\alpha} \rho]), \nu^{-\alpha} \rho, \nu^{-\alpha + 1} \rho, \ldots, \nu^{\lceil \alpha \rceil - \alpha - 1} \rho; \tau^{(2)}) + \\
\pi_1 + \pi_2.
\end{gather*}
\end{proposition}
\begin{proof}
By \cite[Proposition~3.2]{Matic8}, in $R(G)$ we have
\begin{gather*}
\delta([\nu^{a} \rho, \nu^{b} \rho]) \rtimes \delta(\rho, x; \sigma) = L(\delta([\nu^{-b} \rho, \nu^{-a} \rho]); \delta(\rho, x; \sigma)) + \sigma_1 + \\
L(\delta([\nu^{-b} \rho, \nu^{x} \rho]); \delta(\rho, -a; \sigma)) + L(\delta([\nu^{-x} \rho, \nu^{-a} \rho]); \delta(\rho, b; \sigma)),
\end{gather*}
where $\sigma_1$ is the unique common discrete series subrepresentation of both $\delta([\nu^{x} \rho, \nu^{b} \rho]) \rtimes \delta(\rho, a; \sigma)$ and $\delta([\nu^{a} \rho, \nu^{x} \rho]) \rtimes \delta(\rho, b; \sigma)$.

The Aubert duals of $\sigma_1$ and of $L(\delta([\nu^{-b} \rho, \nu^{x} \rho]); \delta(\rho, -a; \sigma))$ can be obtained from Proposition \ref{propmanjeprva}, interchanging the roles of $a$ and $x$. Also, the Aubert dual of $L(\delta([\nu^{-x} \rho, \nu^{-a} \rho]); \delta(\rho, b; \sigma))$ can be obtained from Proposition \ref{propmanjedruga}, interchanging the roles of $b$ and $x$.

It remains to determine the Aubert dual of $L(\delta([\nu^{-b} \rho, \nu^{-a} \rho]); \delta(\rho, x; \sigma))$. First, in the same way as in the previously considered cases we obtain that $\widehat{L(\delta([\nu^{-b} \rho, \nu^{-a} \rho]); \delta(\rho, x; \sigma))}$ is a subrepresentation of
\begin{equation*}
\nu^{-b} \rho \times \cdots \times \nu^{-x-2} \rho \rtimes \widehat{L(\delta([\nu^{-x-1} \rho, \nu^{-a} \rho]); \delta(\rho, x; \sigma))}.
\end{equation*}
Also, if $x > -a+1$, we have
\begin{equation*}
L(\delta([\nu^{-x-1} \rho, \nu^{-a} \rho]); \delta(\rho, x; \sigma)) \hookrightarrow \nu^{x} \rho \times \nu^{x+1} \rho \rtimes L(\delta([\nu^{-x} \rho, \nu^{-a} \rho]); \delta(\rho, x-1; \sigma)),
\end{equation*}
and there is an irreducible subquotient $\pi_1$ of $\nu^{x} \rho \times \nu^{x+1} \rho$ such that $L(\delta([\nu^{-x-1} \rho$, $\nu^{-a} \rho]); \delta(\rho, x; \sigma))$ is a subrepresentation of $\pi_1 \rtimes L(\delta([\nu^{-x} \rho, \nu^{-a} \rho]); \delta(\rho, x-1; \sigma))$.

The induced representation $\delta([\nu^{a} \rho, \nu^{x+1} \rho]) \rtimes \delta(\rho, x; \sigma)$ is a length four representation, again by \cite[Proposition~3.2]{Matic8}. If $\nu^{x+1} \rho \otimes \pi$ is an irreducible constituent of $\mu^{\ast}(\delta([\nu^{a} \rho, \nu^{x+1} \rho]) \rtimes \delta(\rho, x; \sigma))$, using the structural formula we easily obtain that $\pi$ is an irreducible subquotient of $\delta([\nu^{a} \rho, \nu^{x} \rho]) \rtimes \delta(\rho, x; \sigma)$. From \cite[Theorem~4.1]{Mu3} we conclude that $\mu^{\ast}(\delta([\nu^{a} \rho, \nu^{x+1} \rho]) \rtimes \delta(\rho, x; \sigma))$ contains two irreducible constituents of the form $\nu^{x+1} \rho \otimes \pi$, which have to be contained in $\mu^{\ast}(L(\delta([\nu^{-x} \rho, \nu^{-a} \rho]); \delta(\rho, x+1; \sigma)))$ and in $\mu^{\ast}(\sigma_2)$, where $\sigma_2$ is a discrete series subrepresentation of $\delta([\nu^{a} \rho, \nu^{x+1} \rho]) \rtimes \delta(\rho, x; \sigma)$. Thus, $\mu^{\ast}(L(\delta([\nu^{-x-1} \rho, \nu^{-a} \rho]); \delta(\rho, x; \sigma)))$ does not contain irreducible constituents of the form $\nu^{x+1} \rho \otimes \pi$, so $\pi_1 \cong \zeta([\nu^{x} \rho, \nu^{x+1} \rho])$.

This can be used to conclude that the Aubert dual of $L(\delta([\nu^{-x-1} \rho, \nu^{-a} \rho])$; $\delta(\rho, x; \sigma))$ is a subrepresentation of
\begin{equation*}
\delta([\nu^{-x-1} \rho, \nu^{-x} \rho]) \times \cdots \times  \delta([\nu^{a-3} \rho, \nu^{a-2} \rho]) \rtimes \widehat{L(\delta([\nu^{a-2} \rho, \nu^{-a} \rho]); \delta(\rho, -a+1; \sigma))}.
\end{equation*}
Using Proposition \ref{propulag}$(2)$, $(3)$ and $(1)$, respectively, we get
\begin{gather*}
L(\delta([\nu^{a-2} \rho, \nu^{-a} \rho]); \delta(\rho, -a+1; \sigma)) \hookrightarrow \\
\nu^{-a} \rho \times \nu^{-a+1} \rho \times \nu^{-a+2} \rho \rtimes L(\delta([\nu^{a-1} \rho, \nu^{-a-1} \rho]); \delta(\rho, -a; \sigma)).
\end{gather*}
We have already seen that $\mu^{\ast}(L(\delta([\nu^{a-2} \rho, \nu^{-a} \rho]); \delta(\rho, -a+1; \sigma)))$ does not contain an irreducible constituent of the form $\nu^{-a+2} \rho \otimes \pi$. If $\nu^{-a+1} \rho \otimes \pi$ is an irreducible constituent of $\mu^{\ast}(\delta([\nu^{a} \rho, \nu^{-a+2} \rho]) \rtimes \delta(\rho, -a+1; \sigma))$, then $\pi$ is an irreducible subquotient of $\delta([\nu^{a} \rho, \nu^{-a+2} \rho]) \rtimes \delta(\rho, -a; \sigma)$, which is a length two representation. Thus, the Frobenius reciprocity can be used to deduce that $\mu^{\ast}(L(\delta([\nu^{a-2} \rho, \nu^{-a+1} \rho]); \delta(\rho, -a; \sigma)))$ and $\mu^{\ast}(\sigma_3)$, where $\sigma_3$ is a discrete series subrepresentation of $\delta([\nu^{a-2} \rho, \nu^{-a} \rho]) \rtimes \delta(\rho, -a+1; \sigma)$, contain all irreducible constituents of the form $\nu^{-a+1} \rho \otimes \pi$ appearing in $\mu^{\ast}(\delta([\nu^{a} \rho, \nu^{-a+2} \rho]) \rtimes \delta(\rho, -a+1; \sigma))$. So, $L(\delta([\nu^{a-2} \rho, \nu^{-a} \rho]); \delta(\rho, -a+1; \sigma))$ is a subrepresentation of $\zeta([\nu^{-a} \rho, \nu^{-a+2} \rho]) \rtimes L(\delta([\nu^{a-1} \rho, \nu^{-a-1} \rho]); \delta(\rho, -a; \sigma))$. In the same way it can be seen that $\mu^{\ast}(L(\delta([\nu^{a-1} \rho, \nu^{-a-1} \rho]); \delta(\rho, -a; \sigma)))$ does not contain irreducible constituents of the form $\nu^{y} \rho \otimes \pi$ for $\pi \in \{ -a, -a + 1 \}$. Using Lemma \ref{lemaindprva} and continuing in the same way, we get that the Aubert dual of $L(\delta([\nu^{a-1} \rho, \nu^{-a-1} \rho]); \delta(\rho, -a; \sigma))$ is a subrepresentation of
\begin{equation*}
\delta([\nu^{a-2} \rho, \nu^{a} \rho]) \times \cdots \times  \delta([\nu^{-\alpha-2} \rho, \nu^{-\alpha} \rho]) \rtimes \widehat{L(\delta([\nu^{-\alpha-1} \rho, \nu^{\alpha-1} \rho]); \delta(\rho, \alpha; \sigma))}.
\end{equation*}
Let us first consider the case $\alpha = \frac{1}{2}$. Then it can be seen, using the intertwining operators method, that $L(\delta([\nu^{-\frac{3}{2}} \rho, \nu^{-\frac{1}{2}} \rho]); \delta(\rho, \frac{1}{2}; \sigma))$ is a subrepresentation of $\nu^{-\vh} \rho \times \nu^{\vh} \rho \times \nu^{\frac{3}{2}} \rho \rtimes \sigma$. Thus, there is an irreducible subquotient $\pi_1$ of $\nu^{-\vh} \rho \times \nu^{\frac{1}{2}} \rho \times \nu^{\frac{3}{2}} \rho$ such that $L(\delta([\nu^{-\frac{3}{2}} \rho, \nu^{-\vh} \rho]); \delta(\rho, \vh; \sigma))$ is a subrepresentation of $\pi_1 \rtimes \sigma$.

By \cite[Theorem~5.1(ii)]{Mu3}, in $R(G)$ we have
\begin{gather*}
\delta([\nu^{\frac{1}{2}} \rho, \nu^{\frac{3}{2}} \rho]) \rtimes \delta(\rho, \frac{1}{2}; \sigma) = L(\delta([\nu^{-\frac{3}{2}} \rho, \nu^{-\frac{1}{2}} \rho]); \delta(\rho, \frac{1}{2}; \sigma)) + \sigma_4 + \\
L(\delta([\nu^{-\frac{3}{2}} \rho, \nu^{\frac{1}{2}} \rho]); \sigma) + L(\nu^{-\frac{1}{2}} \rho; \delta(\rho, \frac{3}{2}; \sigma)),
\end{gather*}
where $\sigma_4$ is the unique discrete series subrepresentation of 
$\delta([\nu^{\frac{1}{2}} \rho, \nu^{\frac{3}{2}} \rho]) \rtimes \delta(\rho, \frac{1}{2}; \sigma)$.

Since both induced representations $\delta([\nu^{\vh} \rho, \nu^{\frac{3}{2}} \rho]) \rtimes \sigma$ and $\nu^{\frac{1}{2}} \rho \rtimes \delta(\rho, \frac{1}{2}; \sigma)$ are of length two (by \cite[Theorem~5.1]{Mu3}), it follows from the structural formula that $\mu^{\ast}(\delta([\nu^{\vh} \rho, \nu^{\frac{3}{2}} \rho]) \rtimes \delta(\rho, \vh; \sigma))$ contains exactly two irreducible constituents of the form $\nu^{\frac{3}{2}} \rho \otimes \pi$ and exactly two irreducible constituents of the form $\nu^{\vh} \rho \otimes \pi$. Now Frobenius reciprocity and transitivity of the Jacquet modules imply that all irreducible constituents of the form $\nu^{\frac{3}{2}} \rho \otimes \pi$ are contained in $\mu^{\ast}(\sigma_4)$ and in $\mu^{\ast}(L(\nu^{-\frac{1}{2}} \rho; \delta(\rho, \frac{3}{2}; \sigma)))$, while all irreducible constituents of the form $\nu^{\vh} \rho \otimes \pi$ are contained in $\mu^{\ast}(\sigma_4)$ and in
$\mu^{\ast}((L(\delta([\nu^{-\frac{3}{2}} \rho, \nu^{\vh} \rho]); \sigma))$.

Consequently, $\mu^{\ast}(L(\delta([\nu^{-\frac{3}{2}} \rho, \nu^{-\vh} \rho]); \delta(\rho, \vh; \sigma)))$ does not contain irreducible constituents of the form $\nu^{y} \rho \otimes \pi$ for $y \in \{ \vh, \frac{3}{2} \}$.

Thus, it follows that $\pi_1 \cong \zeta([\nu^{-\vh} \rho, \nu^{\frac{3}{2}} \rho])$, so $L(\delta([\nu^{-\frac{3}{2}} \rho, \nu^{-\vh} \rho]); \delta(\rho, \vh; \sigma))$ is a subrepresentation of $\zeta([\nu^{-\vh} \rho, \nu^{\frac{3}{2}} \rho]) \rtimes \sigma$. Now Lemma \ref{lemaindprva} can be used to obtain that the Aubert dual of $L(\delta([\nu^{-\frac{3}{2}} \rho, \nu^{-\vh} \rho]); \delta(\rho, \vh; \sigma))$ is isomorphic to $L(\delta([\nu^{-\frac{3}{2}} \rho, \nu^{\vh} \rho]); \sigma)$.

If $\alpha > \frac{1}{2}$, in the same way as before we deduce that the Aubert dual of $L(\delta([\nu^{-\alpha-1} \rho, \nu^{\alpha-1} \rho]); \delta(\rho, \alpha; \sigma))$ is a subrepresentation of
\begin{equation*}
\delta([\nu^{-\alpha-1} \rho, \nu^{-\alpha+1} \rho]) \rtimes \widehat{L(\delta([\nu^{-\alpha} \rho, \nu^{\alpha-2} \rho]); \sigma)}.
\end{equation*}
If $\alpha = 1$, from \cite[Theorem~3.5]{Matic14} we deduce that $\widehat{L(\delta([\nu^{-\alpha} \rho, \nu^{\alpha-2} \rho]); \sigma)} \cong \delta(\rho, 1; \sigma)$. If $\alpha \geq \frac{3}{2}$, from \cite[Lemma~4.10]{Matic10} we get that $\widehat{L(\delta([\nu^{-\alpha} \rho, \nu^{\alpha-2} \rho]); \sigma)}$ is the unique irreducible subrepresentation of $\nu^{-\alpha+2} \rho \times \cdots \times \nu^{\lceil \alpha \rceil - \alpha - 1} \rho \rtimes \tau^{(2)}$. This ends the proof.
\end{proof}

\begin{proposition}
If $\alpha-1 \leq -a < x$ and $b=x$, in $R(G)$ we have
\begin{gather*}
\zeta([\nu^{-b} \rho, \nu^{-a} \rho]) \rtimes \zeta(\rho, x; \sigma) = \\
L(\nu^{-b} \rho, \nu^{-b} \rho, \ldots, \nu^{a-1} \rho, \nu^{a-1} \rho, \nu^{a} \rho, \nu^{a} \rho, \nu^{a} \rho, \ldots, \nu^{-\alpha} \rho, \nu^{-\alpha} \rho, \nu^{-\alpha} \rho, \\
\nu^{-\alpha+1} \rho, \nu^{-\alpha+1} \rho, \ldots, \nu^{\lceil \alpha \rceil - \alpha - 1} \rho, \nu^{\lceil \alpha \rceil - \alpha - 1} \rho; \tau^{(1)}) + \pi,
\end{gather*}
where
\begin{gather*}
\pi \cong L(\nu^{-b} \rho, \nu^{-b} \rho, \ldots, \nu^{a-2} \rho, \nu^{a-2} \rho, \nu^{a-1} \rho, \delta([\nu^{a-1} \rho, \nu^{a} \rho]), \ldots, \\
\nu^{-\alpha} \rho, \delta([\nu^{-\alpha} \rho, \nu^{-\alpha+1} \rho]), \nu^{-\alpha+1} \rho, \nu^{-\alpha+2} \rho, \nu^{-\alpha+2} \rho, \ldots, \nu^{\lceil \alpha \rceil - \alpha - 1} \rho, \nu^{\lceil \alpha \rceil - \alpha - 1} \rho; \tau^{(1)}),
\end{gather*}
if $\alpha \geq \frac{3}{2}$,
\begin{equation*}
\pi \cong L(\nu^{-b} \rho, \nu^{-b} \rho, \ldots, \nu^{a-2} \rho, \nu^{a-2} \rho, \nu^{a-1} \rho, \delta([\nu^{a-1} \rho, \nu^{a} \rho]), \ldots,
\nu^{-1} \rho, \delta([\nu^{-1} \rho, \rho]); \sigma),
\end{equation*}
if $\alpha = 1$, and
\begin{gather*}
\pi \cong L(\nu^{-b} \rho, \nu^{-b} \rho, \ldots, \nu^{a-2} \rho, \nu^{a-2} \rho, \nu^{a-1} \rho, \delta([\nu^{a-1} \rho, \nu^{a} \rho]), \ldots, \\
\nu^{-\frac{3}{2}} \rho, \delta([\nu^{-\frac{3}{2}} \rho, \nu^{-\frac{1}{2}} \rho]); \tau(\rho, \sigma)),
\end{gather*}
if $\alpha = \frac{1}{2}$.
\end{proposition}
\begin{proof}
In $R(G)$ we have
\begin{gather*}
\delta([\nu^{a} \rho, \nu^{b} \rho]) \rtimes \delta(\rho, b; \sigma) = L(\delta([\nu^{-b} \rho, \nu^{-a} \rho]); \delta(\rho, b; \sigma)) + \tau,
\end{gather*}
where $\tau$ is the unique common irreducible tempered subrepresentation of $\delta([\nu^{a} \rho, \nu^{b} \rho]) \rtimes \delta(\rho, b; \sigma)$ and $\delta([\nu^{-b} \rho, \nu^{b} \rho]) \rtimes \delta(\rho, a; \sigma)$. The Aubert dual of the representation $L(\delta([\nu^{-b} \rho, \nu^{-a} \rho]); \delta(\rho, b; \sigma))$ has been determined in the proof of Proposition \ref{propmanjedruga}, while the Aubert dual of $\tau$ can be obtained from \cite[Theorem~4.16]{Matic10}.
\end{proof}

\begin{proposition}
If $-a < \alpha -2$ and $\alpha - 1 \leq b < x$, in $R(G)$ we have
\begin{gather*}
\zeta([\nu^{-b} \rho, \nu^{-a} \rho]) \rtimes \zeta(\rho, x; \sigma) = \\
L(\nu^{-x} \rho, \ldots, \nu^{-b-1} \rho, \nu^{-b} \rho, \nu^{-b} \rho, \ldots, \nu^{-\alpha} \rho, \nu^{-\alpha} \rho, \nu^{-\alpha+1} \rho, \ldots, \nu^{a-1} \rho, \\
\nu^{a} \rho, \nu^{a} \rho, \ldots, \nu^{\lceil \alpha \rceil - \alpha - 1} \rho, \nu^{\lceil \alpha \rceil - \alpha - 1} \rho; \tau^{(1)}) + \\
L(\nu^{-x} \rho, \ldots, \nu^{-b-2} \rho, \delta([\nu^{-b-1} \rho, \nu^{-b} \rho]), \ldots, \delta([\nu^{-\alpha} \rho, \nu^{-\alpha+1} \rho]), \nu^{-\alpha+2} \rho, \ldots\\
\nu^{a-1} \rho, \nu^{a} \rho, \nu^{a} \rho, \ldots, \nu^{\lceil \alpha \rceil - \alpha - 1} \rho, \nu^{\lceil \alpha \rceil - \alpha - 1} \rho; \tau^{(1)}).
\end{gather*}
If $-a = \alpha -2$ and $\alpha - 1 \leq b < x$, in $R(G)$ we have
\begin{gather*}
\zeta([\nu^{-b} \rho, \nu^{-a} \rho]) \rtimes \zeta(\rho, x; \sigma) = \\
L(\nu^{-x} \rho, \ldots, \nu^{-b-1} \rho, \nu^{-b} \rho, \nu^{-b} \rho, \ldots, \nu^{-\alpha} \rho, \nu^{-\alpha} \rho, \nu^{-\alpha+1} \rho, \nu^{-\alpha+2} \rho, \nu^{-\alpha+2} \rho,\\
\ldots, \nu^{\lceil \alpha \rceil - \alpha - 1} \rho, \nu^{\lceil \alpha \rceil - \alpha - 1} \rho; \tau^{(1)}) + \\
L(\nu^{-x} \rho, \ldots, \nu^{-b-2} \rho, \delta([\nu^{-b-1} \rho, \nu^{-b} \rho]), \ldots, \delta([\nu^{-\alpha} \rho, \nu^{-\alpha+1} \rho]), \\
\nu^{-\alpha+2} \rho, \nu^{-\alpha+2} \rho, \ldots, \nu^{\lceil \alpha \rceil - \alpha - 1} \rho, \nu^{\lceil \alpha \rceil - \alpha - 1} \rho; \tau^{(1)}).
\end{gather*}
\end{proposition}
\begin{proof}
We discuss only the case $-a = \alpha -2$, since the case $-a < \alpha -2$ can be handled in the same way, but more easily. Let us denote by $\sigma_{sp}$ a strongly positive discrete series subrepresentation of $\delta([\nu^{\alpha-1} \rho, \nu^{b} \rho]) \rtimes \delta(\rho, x; \sigma)$ (\cite[Section~4]{Matic3} or Proposition \ref{spds}). Note that we have $\alpha \geq \frac{5}{2}$.

By \cite[Theorem~4.1]{Mu3}, in $R(G)$ we have
\begin{gather*}
\delta([\nu^{-\alpha+2} \rho, \nu^{b} \rho]) \rtimes \delta(\rho, x; \sigma) = L(\delta([\nu^{-b} \rho, \nu^{\alpha-2} \rho]); \delta(\rho, x; \sigma)) + \tau,
\end{gather*}
where $\tau$ is the unique common irreducible (tempered) subrepresentation of induced representations $\delta([\nu^{-\alpha+2} \rho, \nu^{b} \rho]) \rtimes \delta(\rho, x; \sigma)$ and $\delta([\nu^{-\alpha+2} \rho, \nu^{\alpha-2} \rho]) \rtimes \sigma_{sp}$.

Using the same reasoning as in the previously considered cases, we deduce that the Aubert dual of $L(\delta([\nu^{-b} \rho, \nu^{\alpha-2} \rho]); \delta(\rho, x; \sigma))$ is a subrepresentation of
\begin{equation*}
\nu^{-x} \rho \times \cdots \times \nu^{-b-1} \rho \rtimes \widehat{L(\delta([\nu^{-b} \rho, \nu^{\alpha-2} \rho]); \delta(\rho, b; \sigma))}.
\end{equation*}
Since $\nu^{b} \rho \rtimes \delta(\rho, b; \sigma)$ is irreducible, if $b \geq \alpha$ in the same way as before we obtain an embedding
\begin{equation*}
L(\delta([\nu^{-b} \rho, \nu^{\alpha-2} \rho]); \delta(\rho, b; \sigma)) \hookrightarrow \nu^{b} \rho \times \nu^{b} \rho \rtimes L(\delta([\nu^{-b+1} \rho, \nu^{\alpha-2} \rho]); \delta(\rho, b-1; \sigma)),
\end{equation*}
which enables us to deduce that the Aubert dual of $L(\delta([\nu^{-b} \rho, \nu^{\alpha-2} \rho]); \delta(\rho, b; \sigma))$ is an irreducible subrepresentation of
\begin{equation*}
\nu^{-b} \rho \times \nu^{b} \rho \times \cdots \times \nu^{-\alpha} \rho \times \nu^{-\alpha} \rho \rtimes \widehat{L(\delta([\nu^{-\alpha+1} \rho, \nu^{\alpha-2} \rho]); \sigma)},
\end{equation*}
and we have already seen that the Aubert dual of $L(\delta([\nu^{-\alpha+1} \rho, \nu^{\alpha-2} \rho]); \sigma)$ is isomorphic to $L(\nu^{-\alpha+1} \rho, \nu^{-\alpha+2} \rho, \nu^{-\alpha+2} \rho, \ldots, \nu^{\lceil \alpha \rceil - \alpha - 1} \rho, \nu^{\lceil \alpha \rceil - \alpha - 1} \rho; \tau^{(1)})$.

Let us now determine the Aubert dual of $\tau$. If $x > b + 1$, it follows from the classification provided in \cite[Section~4]{Matic3} that $\sigma_{sp}$ is a subrepresentation $\nu^{x} \rho \rtimes \sigma^{(1)}_{sp}$, where $\sigma^{(1)}_{sp}$ is the unique irreducible subrepresentation of $\delta([\nu^{\alpha-1} \rho, \nu^{b} \rho]) \rtimes \delta(\rho, x-1; \sigma)$. Then $\tau$ is a subrepresentation of $\nu^{x} \rho \rtimes \tau_1$, where $\tau_1$ is a common irreducible subrepresentation of both $\delta([\nu^{-\alpha+2} \rho, \nu^{b} \rho]) \rtimes \delta(\rho, x-1; \sigma)$ and $\delta([\nu^{-\alpha+2} \rho, \nu^{\alpha-2} \rho]) \rtimes \sigma^{(1)}_{sp}$. Continuing in this way we obtain that the Aubert dual of $\tau$ is a subrepresentation of
\begin{equation*}
\nu^{-x} \rho \times \cdots \times \nu^{-b-2} \rho \rtimes \widehat{\tau_2},
\end{equation*}
where $\tau_2$ is the unique common irreducible subrepresentation of $\delta([\nu^{-\alpha+2} \rho, \nu^{b} \rho]) \rtimes \delta(\rho, b+1; \sigma)$ and $\delta([\nu^{-\alpha+2} \rho, \nu^{\alpha-2} \rho]) \rtimes \sigma^{(2)}_{sp}$, where $\sigma^{(2)}_{sp}$ is the unique irreducible subrepresentation of $\delta([\nu^{\alpha-1} \rho, \nu^{b} \rho]) \rtimes \delta(\rho, b+1; \sigma)$. Since $\sigma^{(2)}_{sp}$ is a subrepresentation of $\zeta([\nu^{b-1} \rho, \nu^{b} \rho]) \rtimes \sigma^{(3)}_{sp}$, where $\sigma^{(3)}_{sp}$ is the unique irreducible subrepresentation of $\delta([\nu^{\alpha-1} \rho, \nu^{b-1} \rho]) \rtimes \delta(\rho, b; \sigma)$, and $\mu^{\ast}(\sigma^{(3)}_{sp})$ does not contain an irreducible constituent of the form $\nu^{b} \rho \otimes \pi$ by \cite[Theorem~4.6]{Matic4}, we can continue in the same way to obtain that $\widehat{\tau_2}$ is an irreducible subrepresentation of
\begin{equation*}
\delta([\nu^{-b-1} \rho, \nu^{-b} \rho]) \times \cdots \times \delta([\nu^{-\alpha-1} \rho, \nu^{-\alpha} \rho]) \rtimes \widehat{\tau_3},
\end{equation*}
where $\tau_3$ is the unique common irreducible subrepresentation of $\delta([\nu^{-\alpha+2} \rho$, $\nu^{\alpha-1} \rho]) \rtimes \delta(\rho, \alpha; \sigma)$ and $\delta([\nu^{-\alpha+2} \rho, \nu^{\alpha-2} \rho]) \rtimes \sigma^{(4)}_{sp}$, where $\sigma^{(4)}_{sp}$ is the unique irreducible subrepresentation of $\nu^{\alpha-1} \rho \rtimes \delta(\rho, \alpha; \sigma)$.

It follows at once that $\tau_3$ is a subrepresentation of the induced representation $\nu^{\alpha-1} \rho \times \nu^{\alpha} \rho \rtimes \delta([\nu^{-\alpha+2} \rho, \nu^{\alpha-2} \rho]) \rtimes \sigma$. Since $\delta([\nu^{-\alpha+2} \rho, \nu^{\alpha-2} \rho]) \rtimes \sigma$ is irreducible and $\mu^{\ast}(\sigma^{(4)}_{sp})$ does not contain an irreducible constituent of the form $\nu^{\alpha} \rho \otimes \pi$, it follows that $\tau_3$ is a subrepresentation of $\zeta([\nu^{\alpha-1} \rho, \nu^{\alpha} \rho]) \times \delta([\nu^{-\alpha+2} \rho, \nu^{\alpha-2} \rho]) \rtimes \sigma$. Now the rest of the proof follows in the same way as in the previously considered cases. We note that the Aubert dual of $\tau_3$ can also be obtained using \cite[Lemma~4.13,~Lemma~4.15]{Matic10}.
\end{proof}

\begin{proposition} \label{proppunaind}
If $-a < \alpha - 1$ and $x < b$, in $R(G)$ we have
\begin{gather*}
\zeta([\nu^{-b} \rho, \nu^{-a} \rho]) \rtimes \zeta(\rho, x; \sigma) = \\
L(\nu^{-b} \rho, \ldots, \nu^{-x-2} \rho, \delta([\nu^{-x-1} \rho, \nu^{-x} \rho]), \ldots, \delta([\nu^{-\alpha-1} \rho, \nu^{-\alpha} \rho]), \\
\nu^{a} \rho, \ldots, \nu^{\lceil \alpha \rceil - \alpha - 1} \rho; \tau^{(2)}) + \\
L(\nu^{-b} \rho, \ldots, \nu^{-x-1} \rho, \nu^{-x} \rho, \nu^{-x} \rho, \ldots, \nu^{-\alpha} \rho, \nu^{-\alpha} \rho, \nu^{-\alpha+1} \rho, \ldots, \nu^{a-1} \rho, \\ \nu^{a} \rho, \nu^{a} \rho, \ldots, \nu^{\lceil \alpha \rceil - \alpha - 1} \rho, \nu^{\lceil \alpha \rceil - \alpha - 1} \rho; \tau^{(1)}).
\end{gather*}
If $-a = x$, in $R(G)$ we have
\begin{gather*}
\zeta([\nu^{-b} \rho, \nu^{-a} \rho]) \rtimes \zeta(\rho, x; \sigma) = \\
L(\nu^{-b} \rho, \ldots, \nu^{a-2} \rho, \delta([\nu^{a-1} \rho, \nu^{a} \rho]), \nu^{a} \rho, \ldots, \delta([\nu^{-\alpha-1} \rho, \nu^{-\alpha} \rho]), \nu^{-\alpha} \rho\\
\nu^{-\alpha+1} \rho, \ldots, \nu^{\lceil \alpha \rceil - \alpha - 1} \rho; \tau^{(2)}) + \\
L(\nu^{-b} \rho, \ldots, \nu^{a-1} \rho, \nu^{a} \rho, \nu^{a} \rho, \nu^{a} \rho, \ldots, \nu^{-\alpha} \rho, \nu^{-\alpha} \rho, \nu^{-\alpha} \rho,\\
\nu^{-\alpha+1} \rho, \nu^{-\alpha+1} \rho, \ldots, \nu^{\lceil \alpha \rceil - \alpha - 1} \rho, \nu^{\lceil \alpha \rceil - \alpha - 1} \rho; \tau^{(1)}).
\end{gather*}
\end{proposition}
\begin{proof}
If $-a < \alpha - 1$ and $x < b$, in $R(G)$ we have
\begin{gather*}
\delta([\nu^{a} \rho, \nu^{b} \rho]) \rtimes \delta(\rho, x; \sigma) = \\
L(\delta([\nu^{-b} \rho, \nu^{-a} \rho]); \delta(\rho, x; \sigma)) + L(\delta([\nu^{-x} \rho, \nu^{-a} \rho]); \delta(\rho, b; \sigma)).
\end{gather*}
In the same way as in the previously considered cases, we deduce that the Aubert dual of $L(\delta([\nu^{-b} \rho, \nu^{-a} \rho]); \delta(\rho, x; \sigma))$ is a subrepresentation of
\begin{equation*}
\nu^{-b} \rho \times \cdots \times \nu^{-x-2} \rho \rtimes \widehat{L(\delta([\nu^{-x-1} \rho, \nu^{-a} \rho]); \delta(\rho, x; \sigma))}.
\end{equation*}
Since the induced representation $\delta([\nu^{a} \rho, \nu^{x} \rho]) \rtimes \delta(\rho, x; \sigma)$ is irreducible, it follows that $\mu^{\ast}(L(\delta([\nu^{-x-1} \rho, \nu^{-a} \rho]); \delta(\rho, x; \sigma)))$ does not contain an irreducible constituent of the form $\nu^{x+1} \rho \otimes \pi$. Thus, we conclude that
$L(\delta([\nu^{-x-1} \rho, \nu^{-a} \rho])$; $\delta(\rho, x; \sigma))$ is a subrepresentation of $\zeta([\nu^{x} \rho, \nu^{x+1} \rho]) \rtimes L(\delta([\nu^{-x} \rho, \nu^{-a} \rho]); \delta(\rho, x-1; \sigma))$, and, consequently, that the Aubert dual of $L(\delta([\nu^{-x-1} \rho, \nu^{-a} \rho]); \delta(\rho, x; \sigma))$ is a subrepresentation of
\begin{equation*}
\delta([\nu^{-x-1} \rho, \nu^{-x} \rho]) \times \cdots \times \delta([\nu^{-\alpha-1} \rho, \nu^{-\alpha} \rho]) \rtimes \widehat{L(\delta([\nu^{-\alpha} \rho, \nu^{-a} \rho]); \sigma)}.
\end{equation*}
It has been already proved that the Aubert dual of $L(\delta([\nu^{-\alpha} \rho, \nu^{-a} \rho]); \sigma)$ is isomorphic to $L(\nu^{a} \rho, \ldots, \nu^{\lceil \alpha \rceil - \alpha - 1} \rho; \tau^{(2)})$.

Also, the Aubert dual of $L(\delta([\nu^{-x} \rho, \nu^{-a} \rho]); \delta(\rho, b; \sigma))$ is an irreducible subrepresentation of
\begin{equation*}
\nu^{-b} \rho \times \cdots \times \nu^{-x-1} \rho \rtimes \widehat{L(\delta([\nu^{-x} \rho, \nu^{-a} \rho]); \delta(\rho, x; \sigma))}.
\end{equation*}
Since the induced representation $\delta([\nu^{a} \rho, \nu^{x} \rho]) \rtimes \delta(\rho, x; \sigma)$ is irreducible, the Jacquet module of $L(\delta([\nu^{-x} \rho, \nu^{-a} \rho]); \delta(\rho, x; \sigma))$ with respect to the appropriate parabolic subgroup contains
\begin{gather*}
\nu^{x} \rho \otimes \nu^{x} \rho \otimes \cdots \otimes \nu^{\alpha} \rho \otimes \nu^{\alpha} \rho \otimes \nu^{\alpha-1} \rho \otimes \cdots \otimes \nu^{-a+1} \rho \otimes \\
\nu^{-a} \rho \otimes \nu^{-a} \rho \otimes \cdots \otimes \nu^{\alpha - \lceil \alpha \rceil + 1} \rho \otimes \nu^{\alpha - \lceil \alpha \rceil + 1} \rho \otimes \tau',
\end{gather*}
where $\tau' \cong \sigma$ if $a \not\in \mathbb{Z}$ and $\tau' \cong \rho \otimes \sigma$ otherwise. Now, using Lemma \ref{lemaprva} we obtain the Aubert dual of $L(\delta([\nu^{-x} \rho, \nu^{-a} \rho]); \delta(\rho, x; \sigma))$.

If $-a = x$, in $R(G)$ we have
\begin{equation*}
\delta([\nu^{a} \rho, \nu^{b} \rho]) \rtimes \delta(\rho, x; \sigma) =
L(\delta([\nu^{-b} \rho, \nu^{-a} \rho]); \delta(\rho, -a; \sigma)) + \tau,
\end{equation*}
where $\tau$ is the unique irreducible (tempered) common subrepresentation of $\delta([\nu^{a} \rho, \nu^{b} \rho]) \rtimes \delta(\rho, -a; \sigma)$ and $\delta([\nu^{a} \rho, \nu^{-a} \rho]) \rtimes \delta(\rho, b; \sigma)$.

First we have that the Aubert dual of $L(\delta([\nu^{-b} \rho, \nu^{-a} \rho]); \delta(\rho, -a; \sigma))$ is a subrepresentation of
\begin{equation*}
\nu^{-b} \rho \times \cdots \times \nu^{a-2} \rho \rtimes \widehat{L(\delta([\nu^{a-1} \rho, \nu^{-a} \rho]); \delta(\rho, -a; \sigma))}.
\end{equation*}
In the same way as in previously considered cases we deduce that $L(\delta([\nu^{a-1} \rho$, $\nu^{-a} \rho]); \delta(\rho, -a; \sigma))$ is a subrepresentation of
\begin{equation*}
\zeta([\nu^{-a} \rho, \nu^{-a+1} \rho]) \times \nu^{-a} \rho \rtimes L(\delta([\nu^{a} \rho, \nu^{-a-1} \rho]); \delta(\rho, -a-1; \sigma)),
\end{equation*}
and that the Aubert dual of $L(\delta([\nu^{a-1} \rho, \nu^{-a} \rho]); \delta(\rho, -a; \sigma))$ is a subrepresentation of
\begin{equation*}
\delta([\nu^{a-1} \rho, \nu^{a} \rho]) \times \nu^{a} \rho \times \cdots \times \delta([\nu^{-\alpha-1} \rho, \nu^{-\alpha} \rho]) \times \nu^{-\alpha} \rho \rtimes \widehat{L(\delta([\nu^{-\alpha} \rho, \nu^{\alpha-1} \rho]); \sigma)}.
\end{equation*}
It has already been observed that the Aubert dual of $L(\delta([\nu^{-\alpha} \rho, \nu^{\alpha-1} \rho]); \sigma)$ is isomorphic to $L(\nu^{-\alpha+1} \rho, \ldots, \nu^{\lceil \alpha \rceil - \alpha - 1} \rho; \tau^{(2)})$.

In a standard way we obtain that the Aubert dual of $\tau$ is a subrepresentation of
\begin{equation*}
\nu^{-b} \rho \times \cdots \times \nu^{a-1} \rho \rtimes \widehat{\tau'},
\end{equation*}
where $\tau' \cong \delta([\nu^{a} \rho, \nu^{-a} \rho]) \rtimes \delta(\rho, -a; \sigma)$, and now $\widehat{\tau'}$ can be directly obtained using Lemma \ref{lemaprva}. This ends the proof.
\end{proof}

Now we turn our attention to the case $\rho_0 \not\cong \rho$. We omit the proofs, since all the results can be obtained in the same way as in the $\rho_0 \cong \rho$ case, enhanced by Lemma \ref{lemarazl}.

\begin{proposition}
Suppose that $\rho_0 \not\cong \rho$. Then $\zeta([\nu^{-b} \rho_0, \nu^{-a} \rho_0]) \rtimes \zeta(\rho, x; \sigma)$ is irreducible if and only if $b < \beta$. If $b \geq \beta$ and $-a = b$, in $R(G)$ we have
\begin{gather*}
\zeta([\nu^{-b} \rho_0, \nu^{b} \rho_0]) \rtimes \zeta(\rho, x; \sigma) = \\
L(\nu^{-x} \rho, \ldots, \nu^{-\alpha} \rho, \nu^{-b} \rho_0, \nu^{-b} \rho_0, \ldots, \nu^{\lceil \beta \rceil - \beta - 1} \rho_0, \nu^{\lceil \beta \rceil - \beta - 1} \rho_0; \tau^{(1)}) + \\
L(\nu^{-x} \rho, \ldots, \nu^{-\alpha} \rho, \nu^{-b} \rho_0, \nu^{-b} \rho_0, \ldots, \nu^{-\beta} \rho_0, \nu^{-\beta} \rho_0, \nu^{-\beta+1} \rho_0, \ldots, \nu^{\lceil \beta \rceil - \beta - 1} \rho_0; \tau^{(2)}).
\end{gather*}
If $\beta \leq -a < b$, in $R(G)$ we have
\begin{gather*}
\zeta([\nu^{-b} \rho_0, \nu^{-a} \rho_0]) \rtimes \zeta(\rho, x; \sigma) = \\
L(\nu^{-x} \rho, \ldots, \nu^{-\alpha} \rho, \nu^{-b} \rho_0, \ldots, \nu^{a-1} \rho_0, \nu^{a} \rho_0, \nu^{a} \rho_0, \ldots, \nu^{\lceil \beta \rceil - \beta - 1} \rho_0, \nu^{\lceil \beta \rceil - \beta - 1} \rho_0; \tau^{(1)}) + \\
L(\nu^{-x} \rho, \ldots, \nu^{-\alpha} \rho, \nu^{-b} \rho_0, \ldots, \nu^{a-1} \rho_0, \nu^{a} \rho_0, \nu^{a} \rho_0, \ldots, \nu^{-\beta} \rho_0, \nu^{-\beta} \rho_0, \\
\nu^{-\beta+1} \rho_0, \ldots, \nu^{\lceil \beta \rceil - \beta - 1} \rho_0; \tau^{(2)}) + \\
L(\nu^{-x} \rho, \ldots, \nu^{-\alpha} \rho, \nu^{-b} \rho_0, \ldots, \nu^{a-2} \rho_0, \delta([\nu^{a-1} \rho_0, \nu^{a} \rho_0]), \ldots, \delta([\nu^{-\beta-1} \rho_0, \nu^{-\beta} \rho_0]), \\
\nu^{-\beta+1} \rho_0, \ldots, \nu^{\lceil \beta \rceil - \beta - 1} \rho_0; \tau^{(2)}).
\end{gather*}
If $-a < \beta = b$, in $R(G)$ we have
\begin{gather*}
\zeta([\nu^{-b} \rho_0, \nu^{-a} \rho_0]) \rtimes \zeta(\rho, x; \sigma) = \\
L(\nu^{-x} \rho, \ldots, \nu^{-\alpha} \rho, \nu^{-b} \rho_0, \ldots, \nu^{a-1} \rho_0, \nu^{a} \rho_0, \nu^{a} \rho_0, \ldots, \nu^{\lceil \beta \rceil - \beta - 1} \rho_0, \nu^{\lceil \beta \rceil - \beta - 1} \rho_0; \tau^{(1)}) + \\
L(\nu^{-x} \rho, \ldots, \nu^{-\alpha} \rho, \nu^{a} \rho_0, \ldots, \nu^{\lceil \beta \rceil - \beta - 1} \rho_0; \tau^{(2)}).
\end{gather*}
If $-a < \beta  < b$, in $R(G)$ we have
\begin{gather*}
\zeta([\nu^{-b} \rho_0, \nu^{-a} \rho_0]) \rtimes \zeta(\rho, x; \sigma) = \\
L(\nu^{-x} \rho, \ldots, \nu^{-\alpha} \rho, \nu^{-b} \rho_0, \ldots, \nu^{a-1} \rho_0, \nu^{a} \rho_0, \nu^{a} \rho_0, \ldots, \nu^{\lceil \beta \rceil - \beta - 1} \rho_0, \nu^{\lceil \beta \rceil - \beta - 1} \rho_0; \tau^{(1)}) + \\
L(\nu^{-x} \rho, \ldots, \nu^{-\alpha} \rho, \nu^{-b} \rho_0, \ldots, \nu^{-\beta-1} \rho_0, \nu^{a} \rho_0, \ldots, \nu^{\lceil \beta \rceil - \beta - 1} \rho_0; \tau^{(2)}).
\end{gather*}
\end{proposition}

\section{Case $a = \vh$}\label{Section a1/2}

This section is devoted to the case $a = \vh$. Again, we first consider the more complicated case $\rho_0 \cong \rho$, and
let $\tau(\rho_1, \sigma_1)$ be as in the previous section.

Irreducibility criterion is a direct consequence of \cite[Theorem~5.1]{Mu3}:

\begin{proposition}
Degenerate principal series $\zeta([\nu^{-b} \rho, \nu^{-\vh} \rho]) \rtimes \zeta(\rho, x; \sigma)$ is irreducible if and only if one of the following holds:
\begin{itemize}
\item $\alpha > \vh$ and $b = x$,
\item $b < \alpha - 1$.
\end{itemize}
\end{proposition}

The composition factors in other cases are given in the following sequence of propositions.

\begin{proposition}
If $\alpha > \vh$ and $x < b$, in $R(G)$ we have
\begin{gather*}
\zeta([\nu^{-b} \rho, \nu^{-\vh} \rho]) \rtimes \zeta(\rho, x; \sigma) = \\
L(\nu^{-b} \rho, \ldots, \nu^{-x-2} \rho, \delta([\nu^{-x-1} \rho, \nu^{-x} \rho]), \ldots, \delta([\nu^{-\alpha-1} \rho, \nu^{-\alpha} \rho]); \tau^{(2)}) + \\
L(\nu^{-b} \rho, \ldots, \nu^{-x-1} \rho, \nu^{-x} \rho, \nu^{-x} \rho, \ldots, \nu^{-\alpha} \rho, \nu^{-\alpha} \rho, \nu^{-\alpha+1} \rho, \ldots, \nu^{-\vh} \rho; \sigma).
\end{gather*}
\end{proposition}
\begin{proof}
By \cite[Theorem~5.1]{Mu3}, in $R(G)$ we have:
\begin{gather*}
\delta([\nu^{\vh} \rho, \nu^{b} \rho]) \rtimes \delta(\rho, x; \sigma) = \\
L(\delta([\nu^{-b} \rho, \nu^{-\vh} \rho]); \delta(\rho, x; \sigma)) + L(\delta([\nu^{-x} \rho, \nu^{-\vh} \rho]); \delta(\rho, b; \sigma)).
\end{gather*}
First, in a standard way, using the intertwining operators methods, Proposition \ref{propulag}$(1)$ and Lemma \ref{lemaindprva}, we get that the Aubert dual of $L(\delta([\nu^{-b} \rho, \nu^{-\vh} \rho]);$ $\delta(\rho, x; \sigma))$ is a subrepresentation of
\begin{equation*}
\nu^{-b} \rho \times \cdots \times \nu^{-x-2} \rho \rtimes \widehat{L(\delta([\nu^{-x-1} \rho, \nu^{-\vh} \rho]); \delta(\rho, x; \sigma))}.
\end{equation*}
Since $\delta([\nu^{\vh} \rho, \nu^{x} \rho]) \rtimes \delta(\rho, x; \sigma)$ is irreducible, $\nu^{x+1} \rho \otimes \delta([\nu^{\vh} \rho, \nu^{x} \rho]) \rtimes \delta(\rho, x; \sigma)$ is the unique irreducible constituent of the form $\nu^{x+1} \rho \otimes \pi$ appearing in $\mu^{\ast}(\delta([\nu^{\vh} \rho$, $\nu^{x+1} \rho]) \rtimes \delta(\rho, x; \sigma))$, and it is obviously contained in $\mu^{\ast}(L(\delta([\nu^{-x} \rho, \nu^{-\vh} \rho])$; $\delta(\rho, x+1; \sigma))$. This leads to an embedding
\begin{gather*}
L(\delta([\nu^{-x-1} \rho, \nu^{-\vh} \rho]); \delta(\rho, x; \sigma)) \hookrightarrow \\
\zeta([\nu^{x} \rho, \nu^{x+1} \rho]) \rtimes L(\delta([\nu^{-x} \rho, \nu^{-\vh} \rho]); \delta(\rho, x-1; \sigma)),
\end{gather*}
which leads to
\begin{gather*}
L(\delta([\nu^{-x-1} \rho, \nu^{-\vh} \rho]); \delta(\rho, x; \sigma)) \hookrightarrow \\
\delta([\nu^{-x-1} \rho, \nu^{-x} \rho]) \times \cdots \times \delta([\nu^{-\alpha-1} \rho, \nu^{-\alpha} \rho]) \rtimes \widehat{L(\delta([\nu^{-\alpha} \rho, \nu^{-\vh} \rho]); \sigma)},
\end{gather*}
and by \cite[Theorem~3.5]{Matic14} the Aubert dual of $L(\delta([\nu^{-\alpha} \rho, \nu^{-\vh} \rho]); \sigma)$ is isomorphic to $\tau^{(2)}$.

Using Proposition \ref{propulag}$(3)$ and Lemma \ref{lemaindprva}, we deduce that the Aubert dual of $L(\delta([\nu^{-x} \rho, \nu^{-\vh} \rho]); \delta(\rho, b; \sigma))$ is a subrepresentation of
\begin{equation*}
\nu^{-b} \rho \times \cdots \times \nu^{-x-1} \rho \rtimes \widehat{L(\delta([\nu^{-x} \rho, \nu^{-\vh} \rho]); \delta(\rho, x; \sigma))}.
\end{equation*}
Now by irreducibility of $\delta([\nu^{\vh} \rho, \nu^{x} \rho]) \rtimes \delta(\rho, x; \sigma)$, the rest of the proof follows in the same way as in the proof of Proposition \ref{proppunaind}.
\end{proof}

\begin{proposition}
If $\alpha > \vh$ and $\alpha-1 \leq b < x$, in $R(G)$ we have
\begin{gather*}
\zeta([\nu^{-b} \rho, \nu^{-\vh} \rho]) \rtimes \zeta(\rho, x; \sigma) = \\
L(\nu^{-x} \rho, \ldots, \nu^{-b-1} \rho, \nu^{-b} \rho, \nu^{-b} \rho, \ldots, \nu^{-\alpha} \rho, \nu^{-\alpha} \rho, \nu^{-\alpha+1} \rho, \ldots, \nu^{-\vh} \rho; \sigma) + \\
L(\nu^{-x} \rho, \ldots, \nu^{-b-2} \rho, \delta([\nu^{-b-1} \rho, \nu^{-b} \rho]), \ldots, \delta([\nu^{-\alpha} \rho, \nu^{-\alpha+1} \rho]), \nu^{-\alpha+2} \rho, \ldots, \nu^{-\vh} \rho; \sigma).
\end{gather*}
\end{proposition}
\begin{proof}
By \cite[Theorem~5.1]{Mu3}, in $R(G)$ we have:
\begin{gather*}
\delta([\nu^{\vh} \rho, \nu^{b} \rho]) \rtimes \delta(\rho, x; \sigma) = \\
L(\delta([\nu^{-b} \rho, \nu^{-\vh} \rho]); \delta(\rho, x; \sigma)) + L(\delta([\nu^{-\alpha+2} \rho, \nu^{-\vh} \rho]); \sigma_{sp}),
\end{gather*}
where $\sigma_{sp}$ is the unique irreducible (strongly positive) subrepresentation of $\delta([\nu^{\alpha - 1} \rho, \nu^{b} \rho]) \rtimes \delta(\rho, x; \sigma)$.

The Aubert dual of $L(\delta([\nu^{-b} \rho, \nu^{-\vh} \rho]); \delta(\rho, x; \sigma))$ can be obtained following the same lines as in the proof of the previous proposition, interchanging the roles of $b$ and $x$. On the other hand, the Aubert dual of $L(\delta([\nu^{-\alpha+2} \rho, \nu^{-\vh} \rho]); \sigma_{sp})$ can be determined in the same way as in the proof of Proposition \ref{propspprva}.
\end{proof}

\begin{proposition}
If $\alpha = \vh$ and $x < b$, in $R(G)$ we have
\begin{gather*}
\zeta([\nu^{-b} \rho, \nu^{-\vh} \rho]) \rtimes \zeta(\rho, x; \sigma) = \\
L(\nu^{-b} \rho, \ldots, \nu^{-x-1} \rho, \nu^{-x} \rho, \nu^{-x} \rho, \ldots, \nu^{-\frac{1}{2}} \rho, \nu^{-\frac{1}{2}} \rho; \sigma) + \\
L(\nu^{-b} \rho, \ldots, \nu^{-x-2} \rho, \delta([\nu^{-x-1} \rho, \nu^{-x} \rho]), \ldots, \delta([\nu^{-\frac{5}{2}} \rho, \nu^{-\frac{3}{2}} \rho]), \delta([\nu^{-\frac{3}{2}} \rho, \nu^{\frac{1}{2}} \rho]); \sigma) + \\
L(\nu^{-b} \rho, \ldots, \nu^{-x-2} \rho, \delta([\nu^{-x-1} \rho, \nu^{-x} \rho]), \ldots, \delta([\nu^{-\frac{3}{2}} \rho, \nu^{-\frac{1}{2}} \rho]); \delta(\rho, \vh; \sigma)) + \\
L(\nu^{-b} \rho, \ldots, \nu^{-x-1} \rho, \nu^{-x} \rho, \nu^{-x} \rho, \ldots, \nu^{-\frac{3}{2}} \rho, \nu^{-\frac{3}{2}} \rho; \tau(\rho; \sigma)).
\end{gather*}
\end{proposition}
\begin{proof}
By \cite[Theorem~5.1]{Mu3}, in $R(G)$ we have:
\begin{gather*}
\delta([\nu^{\vh} \rho, \nu^{b} \rho]) \rtimes \delta(\rho, x; \sigma) =
L(\delta([\nu^{-b} \rho, \nu^{-\vh} \rho]); \delta(\rho, x; \sigma)) + \sigma_{ds} + \\
L(\delta([\nu^{-x} \rho, \nu^{-\vh} \rho]); \delta(\rho, b; \sigma)) + L(\delta([\nu^{-b} \rho, \nu^{x} \rho]); \sigma),
\end{gather*}
where $\sigma_{ds}$ is the unique common irreducible (discrete series) subrepresentation of both $\delta([\nu^{\vh} \rho, \nu^{b} \rho]) \rtimes \delta(\rho, x; \sigma)$ and $\delta([\nu^{-x} \rho, \nu^{b} \rho]) \rtimes \sigma$. Note that $\widehat{\sigma_{ds}}$ has been determined in \cite[Theorem~5.2.(i)]{Matic10}.

Let us now determine the Aubert dual of $L(\delta([\nu^{-x} \rho, \nu^{-\vh} \rho]); \delta(\rho, b; \sigma))$. In a standard way we conclude that it is a subrepresentation of
\begin{equation*}
\nu^{-b} \rho \times \cdots \times \nu^{-x-1} \rho \rtimes \widehat{L(\delta([\nu^{-x} \rho, \nu^{-\vh} \rho]); \delta(\rho, x; \sigma))}.
\end{equation*}
For $x \geq \frac{3}{2}$, we have the following embeddings and isomorphisms:
\begin{align*}
L(\delta([\nu^{-x} \rho, \nu^{-\vh} \rho]); \delta(\rho, x; \sigma)) & \hookrightarrow \delta([\nu^{-x+1} \rho, \nu^{-\vh} \rho]) \times \nu^{-x} \rho \rtimes \delta(\rho, x; \sigma) \\
& \cong \nu^{x} \rho \times \delta([\nu^{-x+1} \rho, \nu^{-\vh} \rho]) \rtimes \delta(\rho, x; \sigma) \\
& \hookrightarrow \nu^{x} \rho \times \nu^{x} \rho \times \delta([\nu^{-x+1} \rho, \nu^{-\vh} \rho]) \rtimes \delta(\rho, x-1; \sigma),
\end{align*}
which enable us to conclude that $L(\delta([\nu^{-x} \rho, \nu^{-\vh} \rho]); \delta(\rho, x; \sigma))$ is a subrepresentation of $\nu^{x} \rho \times \nu^{x} \rho \rtimes L(\delta([\nu^{-x+1} \rho, \nu^{-\vh} \rho]); \delta(\rho, x-1; \sigma))$. Thus, repeating these arguments and using Lemma \ref{lemainddruga}, we get that the Aubert dual of $L(\delta([\nu^{-x} \rho, \nu^{-\vh} \rho]); \delta(\rho, x; \sigma))$ is a subrepresentation of
\begin{equation*}
\nu^{-x} \rho \times \nu^{-x} \rho \times \cdots \times \nu^{-\frac{3}{2}} \rho \times \nu^{-\frac{3}{2}} \rho \rtimes \widehat{L(\nu^{-\vh} \rho; \delta(\rho, \vh; \sigma))},
\end{equation*}
and it has already been observed that $\widehat{L(\nu^{-\vh} \rho; \delta(\rho, \vh; \sigma))} \cong \tau(\rho; \sigma)$.

Now we analyze the Aubert duals of representations $L(\delta([\nu^{-b} \rho, \nu^{x} \rho]); \sigma)$ and $L(\delta([\nu^{-b} \rho, \nu^{-\vh} \rho]); \delta(\rho, x; \sigma))$. Using the same arguments as before, we obtain the following embeddings:
\begin{gather*}
\widehat{L(\delta([\nu^{-b} \rho, \nu^{x} \rho]); \sigma)} \hookrightarrow
\nu^{-b} \rho \times \cdots \times \nu^{-x-2} \rho \rtimes \widehat{L(\delta([\nu^{-x-1} \rho, \nu^{x} \rho]); \sigma)}, \\
\widehat{L(\delta([\nu^{-b} \rho, \nu^{-\vh} \rho]); \delta(\rho, x; \sigma))} \hookrightarrow \\
\nu^{-b} \rho \times \cdots \times \nu^{-x-2} \rho \rtimes \widehat{L(\delta([\nu^{-x-1} \rho, \nu^{-\vh} \rho]); \delta(\rho, x; \sigma))}.
\end{gather*}
Since $\delta([\nu^{\vh} \rho, \nu^{x} \rho]) \rtimes \delta(\rho, x; \sigma)$ is a length two representation by \cite[Theorem~5.1]{Mu3}, it follows at once from the structural formula that $\mu^{\ast}(\delta([\nu^{\vh} \rho, \nu^{x+1} \rho]) \rtimes \delta(\rho, x; \sigma))$ contains exactly two irreducible constituents of the form $\nu^{x+1} \rho \otimes \pi$, which have to be contained in $\mu^{\ast}(L(\delta([\nu^{-x} \rho, \nu^{-\vh} \rho]); \delta(\rho, x+1; \sigma)))$ and in $\mu^{\ast}(\sigma'_{ds})$, where $\sigma'_{ds}$ is the unique discrete series subquotient of $\delta([\nu^{\vh} \rho, \nu^{x+1} \rho]) \rtimes \delta(\rho, x; \sigma)$. Thus, neither $\mu^{\ast}(L(\delta([\nu^{-x-1} \rho, \nu^{x} \rho]); \sigma))$, nor $\mu^{\ast}(L(\delta([\nu^{-x-1} \rho, \nu^{-\vh} \rho])$; $\delta(\rho, x; \sigma)))$ contains irreducible constituent of the form $\nu^{x+1} \rho \otimes \pi$. This leads to an embedding
\begin{equation*}
L(\delta([\nu^{-x-1} \rho, \nu^{x} \rho]); \sigma) \hookrightarrow \zeta([\nu^{x} \rho, \nu^{x+1} \rho]) \rtimes L(\delta([\nu^{-x} \rho, \nu^{x-1} \rho]); \sigma)
\end{equation*}
and, if $x \geq \frac{3}{2}$, to an embedding
\begin{gather*}
L(\delta([\nu^{-x-1} \rho, \nu^{-\vh} \rho]); \delta(\rho, x; \sigma)) \hookrightarrow \\
\zeta([\nu^{x} \rho, \nu^{x+1} \rho]) \rtimes L(\delta([\nu^{-x} \rho, \nu^{-\vh} \rho]); \delta(\rho, x-1; \sigma)).
\end{gather*}
Using Lemma \ref{lemaindtreca} and repeating the same arguments, we obtain
\begin{gather*}
\widehat{L(\delta([\nu^{-x-1} \rho, \nu^{x} \rho]); \sigma)} \hookrightarrow \\
\delta([\nu^{-x-1} \rho, \nu^{-x} \rho]) \times \cdots \times \delta([\nu^{-\frac{3}{2}} \rho, \nu^{-\vh} \rho]) \rtimes \widehat{L(\nu^{-\frac{1}{2}} \rho; \sigma)},
\end{gather*}
and
\begin{gather*}
\widehat{L(\delta([\nu^{-x-1} \rho, \nu^{-\vh} \rho]); \delta(\rho, x; \sigma))} \hookrightarrow \\
\delta([\nu^{-x-1} \rho, \nu^{-x} \rho]) \times \cdots \times \delta([\nu^{-\frac{5}{2}} \rho, \nu^{-\frac{3}{2}} \rho]) \rtimes \widehat{L(\delta([\nu^{-\frac{3}{2}} \rho, \nu^{-\vh} \rho]); \delta(\rho, \vh; \sigma))},
\end{gather*}
We have already seen that $\widehat{L(\nu^{-\frac{1}{2}} \rho; \sigma)} \cong \delta(\rho, \vh; \sigma)$ and that the Aubert dual of $L(\delta([\nu^{-\frac{3}{2}} \rho, \nu^{-\vh} \rho]); \delta(\rho, \vh; \sigma))$ is isomorphic to $L(\delta([\nu^{-\frac{3}{2}} \rho, \nu^{\vh} \rho]); \sigma)$. This ends the proof.
\end{proof}

\begin{proposition}
If $\alpha = \vh$ and $b \leq x$, in $R(G)$ we have
\begin{gather*}
\zeta([\nu^{-b} \rho, \nu^{-\vh} \rho]) \rtimes \zeta(\rho, x; \sigma) = \\
L(\nu^{-x} \rho, \ldots, \nu^{-b-1} \rho, \nu^{-b} \rho, \nu^{-b} \rho, \ldots, \nu^{-\frac{3}{2}} \rho, \nu^{-\frac{3}{2}} \rho; \tau(\rho; \sigma)) + \\
L(\nu^{-x} \rho, \ldots, \nu^{-b-1} \rho, \nu^{-b} \rho, \nu^{-b} \rho, \ldots, \nu^{-\frac{1}{2}} \rho, \nu^{-\frac{1}{2}} \rho; \sigma).
\end{gather*}
\end{proposition}
\begin{proof}
By \cite[Theorem~5.1]{Mu3}, in $R(G)$ we have:
\begin{equation*}
\delta([\nu^{\vh} \rho, \nu^{b} \rho]) \rtimes \delta(\rho, x; \sigma) =
L(\delta([\nu^{-b} \rho, \nu^{-\vh} \rho]); \delta(\rho, x; \sigma)) + \tau,
\end{equation*}
where $\tau$ is the unique common irreducible (tempered) subrepresentation of both $\delta([\nu^{\vh} \rho, \nu^{b} \rho]) \rtimes \delta(\rho, x; \sigma)$ and $\delta([\nu^{-b} \rho, \nu^{x} \rho]) \rtimes \sigma$. Note that $\tau$ is a discrete series if $b < x$. The Aubert duals of $L(\delta([\nu^{-b} \rho, \nu^{-\vh} \rho]); \delta(\rho, x; \sigma))$ and $\tau$ can be obtained in the same way as in the proof of the previous proposition (and in the proof of \cite[Theorem~5.2.(i)]{Matic10}), interchanging the roles of $b$ and $x$.
\end{proof}

The remaining case is covered in the following proposition, a detailed verification being left to the reader.

\begin{proposition}
Suppose that $\rho_0 \not\cong \rho$. Then the degenerate principal series $\zeta([\nu^{-b} \rho_0, \nu^{-\vh} \rho_0]) \rtimes \zeta(\rho, x; \sigma)$ is irreducible if and only if $b < \beta$. If $b \geq \beta$, in $R(G)$ we have
\begin{gather*}
\zeta([\nu^{-b} \rho_0, \nu^{-\vh} \rho_0]) \rtimes \zeta(\rho, x; \sigma) =
L(\nu^{-x} \rho, \ldots, \nu^{-\alpha} \rho, \nu^{-b} \rho, \ldots, \nu^{-\vh} \rho; \sigma) + \\
L(\nu^{-x} \rho, \ldots, \nu^{-\alpha} \rho, \nu^{-b} \rho, \ldots, \nu^{-\beta-1} \rho; \tau^{(2)}).
\end{gather*}
\end{proposition}

\section{The odd $GSpin$ case}
\label{GSpin}

In this section we consider the odd $GSpin$ case. 

\begin{remark}
All the propositions in Sections \ref{Section beta0} -- \ref{Section a1/2} are valid for the odd $GSpin$ case with exactly the same statements. More precisely, all the arguments used in \cite{Matic14, Matic10, Mu3} (except \cite[Theorem 2.1]{Mu3}), as well as those used in the previous sections, can be directly carried out to the odd $GSpin$ case, since they completely rely on properties of the Aubert involution which hold for general reductive groups, the structural formula and classifications of discrete series provided for the odd $GSpin$ groups in \cite{Kim1, KimMatic} (see also Lemma \ref{osn} for the structure formula for odd $GSpin$ groups). In the following, we will comment on the generalizations of the results in \cite{Mu3} to odd $GSpin$ groups and give the proof for the odd $GSpin$ case of \cite[Theorem 2.1]{Mu3}.  
\end{remark}

Let us first recall the definition of odd $GSpin$ groups.
Let $\nu_m$ be the $m \times m$ matrix with ones on the second diagonal and zeros elsewhere. Let $J_{2m} = \begin{pmatrix}
0 & \nu_m\\
-\nu_m & 0 
\end{pmatrix}$. Then the similitude symplectic groups are defined as follows:
\begin{align*}
GSp(2n, F) & = \{g\in GL(2n, F) : {}^tg J_{2n} g = \lambda(g) J_{2n} \text{ for some } \lambda(g) \in F^*\}.
\end{align*}
Let $T = \{t=\mathrm{diag}(t_1, \ldots, t_n, a t_n^{-1}, \ldots, a t_1^{-1}): t_i, a \in F^*\},$ then $T$ is a maximal torus for $GSp(2n, F)$. For $t=\mathrm{diag}(t_1, \ldots, t_n, a t_n^{-1}, \ldots, a t_1^{-1}) \in T$, let $e_0(t)=a$, and let $e_i(t)=t_i$ for $i=1,\ldots,n$. Let $X=\operatorname{Hom}(T, F^*)$ be the character lattice of $T$. Then
$X=\mathbb{Z} e_0 \oplus \mathbb{Z} e_1 \oplus \cdots \oplus \mathbb{Z} e_n$. 
Let $X^{\vee} = \operatorname{Hom}(F^*, T)$ be the cocharacter lattice of $X$, and let $\{e_0^*, e_1^*, \ldots, e_n^*\}$ be the basis of $X^{\vee}$ dual to the basis $\{e_0, e_1, \ldots, e_n\}$ of $X$. Then $X^{\vee}=\mathbb{Z} e_0^* \oplus \mathbb{Z} e_1^* \oplus \cdots \oplus \mathbb{Z} e_n^*$. Let $\Delta  = \{ e_i-e_{i+1}, i=1,\ldots,n-1, 2e_n-e_0\}, \Delta^{\vee}  = \{e_i^*-e_{i+1}^*,i=1,\ldots,n-1, e_n^*\}. $ Then the root datum of $GSp(2n)$ is $(X, \Delta, X^{\vee}, \Delta^{\vee})$. 

\begin{definition}
$GSpin(2n+1, F)$ is $F$-points of the unique split $F$-group having root datum $(X^{\vee}, \Delta^{\vee}, X, \Delta)$ which is dual to that of $GSp(2n, F)$.
\end{definition}

\begin{remark}
Let $Spin(2n+1, F)$ be the double covering of special orthogonal group $SO(2n+1, F)$. Then by \cite[Proposition 2.2]{Asgari1}, the derived group of the split $GSpin(2n+1, F)$ is $Spin(2n+1, F)$ and $GSpin(2n+1, F)$ is isomorphic to 
$$(GL(1,F) \times Spin(2n+1,F))/\{(1,1), (-1,c)\},$$
where $c = (2e_n-e_0)(-1)$.
\end{remark}

We now briefly summarize the main results in \cite{Mu3}. Let $H_n$ be either a symplectic group or special odd orthogonal group defined over a non-archimedean local field $F$ of characteristic different from $2$, having split rank $n$. In \cite{Mu3}, Mui\'c studies the reducibility of $\delta \rtimes \sigma$, where $\sigma$ is a strongly positive representation in $H_n(F)$ and $\delta:=\delta([\nu^{-l_1}\rho, \nu^{l_2}\rho]) $ is an irreducible essentially square integrable representation of $GL_m(F)$ (Here, $\rho$ is an irreducible unitary cuspidal representation of $GL(F)$ and $l_1, l_2 \in \mathbb{R}$ is such that $l_1 + l_2 \in \mathbb{Z}_{\geq 0}$). Mui\'c, in \cite{Mu3},  further describes the composition series of $\delta \rtimes \sigma$ if it is reducible. Chapter 3, Chapter 4, and Chapter 5 in \cite{Mu3} describe the cases $l_1 \leq -1, l_1 \geq 0,$ and $l_1 = -1/2$ (Proposition 3.1, Theorem 4.1, and Theorem 5.1), respectively. The main ingredients for the proofs of those propositions and theorems are Tadi\'c's structure formula for $H_n$ \cite{Tad5} 
 (he mainly uses the information from $GL$ cuspidal part in the Jacquet modules of the representations) 
and the classification of discrete series of $H_n$ \cite{MT1}. All those ingredients are now available for odd $GSpin$ groups (Lemma \ref{osn} and \cite{KimMatic}).
However, we note that the proof of \cite[Theorem~2.1]{Mu3} can not be applied to the $GSpin$ groups. We will reprove this theorem below (Theorem \ref{GSpinTheorem2.1}), in the case which we use when determining the composition factors of the degenerate principal series. Then, for odd $GSpin$ groups, all the results in Chapters 3, 4, and 5 in \cite{Mu3}, together with the correction of \cite[Theorem~4.1.(iv),~Lemma~4.9]{Mu3} obtained in \cite[Proposition~3.2]{Matic8}, follow in the same way as in those two papers.
Therefore, our results on the composition factors of the degenerate principal series also hold in the odd $GSpin$ case. 


\begin{remark}
To prove \cite[Theorem~2.1]{Mu3}, two lemmas (\cite[Lemma~2.1,~2.2]{Mu3}: description of non-tempered subquotients and tempered but non-square integrable subquotients of generalized principal series) are needed. The main ingredients in the proofs of those lemmas are again Tadi\'c's structure formula (especially the information about $GL$ cuspidal support), Casselman's square-integrability criterion, and classification of discrete series representations, which all can be applied directly to $GSpin(2n+1, F)$, so we skip the proofs of those lemmas for $GSpin(2n+1,F)$.
\end{remark}


Recall that $\alpha$ (resp. $\beta$) is the reducibility point of $\rho$ (resp. $\rho_0$) and $\sigma$, i.e., $\nu^{s} \rho \rtimes \sigma$ (resp. $\nu^{s} \rho_0 \rtimes \sigma$) is irreducible if and only if $s \not\in \{ \alpha, -\alpha \}$ (resp. $s \not\in \{ \beta, -\beta \})$.

\begin{thrm}\label{GSpinTheorem2.1}
Suppose that $\sigma$ is an irreducible cuspidal representation of $GSpin(2n+1, F)$, and that one of the following holds:
\begin{enumerate}[(1)]
\item $\rho_0 \not\cong \rho$, $\beta \leq -a < b$, and $b - \beta \in \mathbb{Z}$,
\item $\rho_0 \cong \rho$, $b > -a > x$, and $b - \alpha \in \mathbb{Z}$,
\item $\rho_0 \cong \rho$, $\alpha - 1 \leq -a < b < x$, $-a \geq 0$, and $b - \alpha \in \mathbb{Z}$.
\end{enumerate}
Then in an appropriate Grothendieck group we have
\begin{gather*}
\delta([\nu^{a} \rho_0, \nu^{b} \rho_0]) \rt \delta(\rho, x; \sigma) = L(\delta([\nu^{-b} \rho_0, \nu^{-a} \rho_0]); \delta(\rho, x; \sigma)) + \sigma^{(1)}_{ds} + \sigma^{(2)}_{ds},
\end{gather*}
where $\sigma^{(1)}_{ds}$ and $\sigma^{(2)}_{ds}$ are mutually non-isomorphic discrete series subrepresentations of $\delta([\nu^{a} \rho_0, \nu^{b} \rho_0]) \rt \delta(\rho, x; \sigma)$.
\end{thrm}
\begin{proof}
We prove only the part $(3)$, other parts can be proved in the same way, but more easily. It can be seen in the same way as in the proof of \cite[Theorem~2.1]{Mu3} that $L(\delta([\nu^{-b} \rho_0, \nu^{-a} \rho_0]); \delta(\rho, x; \sigma))$ is the unique non-tempered irreducible subquotient of $\delta([\nu^{a} \rho_0, \nu^{b} \rho_0]) \rt \delta(\rho, x; \sigma)$. Also, representations $\sigma^{(1)}_{ds}$ and $\sigma^{(2)}_{ds}$ have been constructed in \cite[Theorem~3.14]{KimMatic}. Let us prove that there are no other irreducible tempered subquotients of $\delta([\nu^{a} \rho_0, \nu^{b} \rho_0]) \rt \delta(\rho, x; \sigma)$.

Let $\pi$ denote an irreducible tempered subquotient of $\delta([\nu^{a} \rho_0, \nu^{b} \rho_0]) \rt \delta(\rho, x; \sigma)$. From the cuspidal support considerations one can conclude that $\pi$ has to be square-integrable and non-strongly positive. Thus, by the classification given in \cite{KimMatic}, if $\alpha \geq 2$, $\pi$ can be written as a subrepresentation of one of the following induced representations:
\begin{equation*}
\delta([\nu^{a} \rho, \nu^{b} \rho]) \rtimes \delta(\rho, x; \sigma), \delta([\nu^{-b} \rho, \nu^{x} \rho]) \rtimes \delta(\rho, -a; \sigma), \delta([\nu^{-\alpha + 2} \rho, \nu^{-a} \rho]) \rtimes \sigma_{sp},
\end{equation*}
where $\sigma_{sp}$ stands for the unique irreducible subrepresentation of $\delta([\nu^{\alpha-1} \rho, \nu^{b} \rho])$ $\rtimes \delta(\rho, x; \sigma)$.
Thus, $\mu^{\ast}(\pi)$ contains one of the following irreducible constituents:
\begin{equation*}
\delta([\nu^{a} \rho, \nu^{b} \rho]) \otimes \delta(\rho, x; \sigma), \delta([\nu^{-b} \rho, \nu^{x} \rho]) \otimes \delta(\rho, -a; \sigma), \delta([\nu^{-\alpha + 2} \rho, \nu^{-a} \rho]) \otimes \sigma_{sp}.
\end{equation*}

If $\alpha < 2$, $\pi$ can be written as a subrepresentation of one of the following induced representations:
\begin{equation*}
\delta([\nu^{a} \rho, \nu^{b} \rho]) \rtimes \delta(\rho, x; \sigma), \delta([\nu^{-b} \rho, \nu^{x} \rho]) \rtimes \delta(\rho, -a; \sigma),
\end{equation*}
and $\mu^{\ast}(\pi)$ contains one of the following irreducible constituents:
\begin{equation*}
\delta([\nu^{a} \rho, \nu^{b} \rho]) \otimes \delta(\rho, x; \sigma), \delta([\nu^{-b} \rho, \nu^{x} \rho]) \otimes \delta(\rho, -a; \sigma).
\end{equation*}

By \cite[Theorem~3.14]{KimMatic}, only irreducible subrepresentations of $\delta([\nu^{a} \rho, \nu^{b} \rho])$ $\rt \delta(\rho, x; \sigma)$ are $\sigma^{(1)}_{ds}$ and $\sigma^{(2)}_{ds}$. Also, it is easy to see, using the odd $GSpin$ version of the structural formula given in \cite{Kim1}, together with the classification of strongly positive discrete series, that $\delta([\nu^{-b} \rho, \nu^{x} \rho]) \otimes \delta(\rho, -a; \sigma)$ appears with multiplicity one in $\mu^{\ast}(\delta([\nu^{a} \rho, \nu^{b} \rho]) \rtimes \delta(\rho, x; \sigma))$, and that $\delta([\nu^{-\alpha + 2} \rho, \nu^{-a} \rho]) \otimes \sigma_{sp}$ also appears with multiplicity one in $\mu^{\ast}(\delta([\nu^{a} \rho, \nu^{b} \rho]) \rtimes \delta(\rho, x; \sigma))$ if $\alpha \geq 2$.

Let $\tau_i$, for $i \in \{ 1, 2 \}$, denote an irreducible tempered subrepresentation of $\delta([\nu^{a} \rho, \nu^{-a} \rho]) \rtimes \delta(\rho, x; \sigma)$ such that $\sigma^{(i)}_{ds}$ is the unique irreducible subrepresentation of $\delta([\nu^{-a+1} \rho, \nu^{b} \rho]) \rtimes \tau_i$. By \cite[Section~4]{Tad6}, there is a unique $j \in \{ 1, 2 \}$ such that $\tau_j$ is a subrepresentation of $\delta([\nu^{-a+1} \rho, \nu^{x} \rho]) \times \delta([\nu^{a} \rho, \nu^{-a} \rho]) \rtimes \delta(\rho, -a; \sigma)$. It follows from the proof of \cite[Theorem~3.15]{KimMatic} that $\sigma^{(j)}_{ds}$ is a subrepresentation of $\delta([\nu^{-b} \rho, \nu^{x} \rho]) \rtimes \delta(\rho, -a; \sigma)$, so $\mu^{\ast}(\sigma^{(j)}_{ds})$ contains $\delta([\nu^{-b} \rho, \nu^{x} \rho]) \otimes \delta(\rho, -a; \sigma)$.

Similarly, if $\alpha \geq 2$, then there is a unique $k \in \{ 1, 2 \}$ such that $\tau_k$ is a subrepresentation of $\delta([\nu^{\alpha-1} \rho, \nu^{-a} \rho]) \times \delta([\nu^{\alpha-1} \rho, \nu^{-a} \rho]) \times \delta([\nu^{-\alpha+2} \rho, \nu^{\alpha-2} \rho]) \rtimes \delta(\rho, x; \sigma)$. It follows from the proof of \cite[Theorem~3.15]{KimMatic} that $\sigma^{(k)}_{ds}$ is a subrepresentation of $\delta([\nu^{-\alpha + 2} \rho, \nu^{-a} \rho]) \rtimes \sigma_{sp}$. Frobenius reciprocity implies that $\mu^{\ast}(\sigma^{(k)}_{ds})$ contains $\delta([\nu^{-\alpha + 2} \rho, \nu^{-a} \rho]) \otimes \sigma_{sp}$.

From the multiplicities of $\delta([\nu^{a} \rho, \nu^{b} \rho]) \otimes \delta(\rho, x; \sigma)$, $\delta([\nu^{-b} \rho, \nu^{x} \rho]) \otimes \delta(\rho, -a; \sigma)$, and $\delta([\nu^{-\alpha + 2} \rho, \nu^{-a} \rho]) \otimes \sigma_{sp}$ in $\mu^{\ast}(\delta([\nu^{a} \rho, \nu^{b} \rho]) \rtimes \delta(\rho, x; \sigma))$, we conclude that $\pi$ is isomorphic either to $\sigma^{(1)}_{ds}$ or to $\sigma^{(2)}_{ds}$, and the theorem is proved.
\end{proof}

\bibliographystyle{siam}
\bibliography{Literatura}

\begin{flushleft}
{Yeansu Kim\\
Department of Mathematics Education\\
Chonnam National University\\
77 Yongbong-ro, Buk-gu, Gwangju city, South Korea\\
E-mail: ykim@chonnam.ac.kr}
\\
{Baiying Liu \\
Department of Mathematics\\
Purdue University\\
West Lafayette, IN, USA 47907\\
E-mail: liu2053@purdue.edu}
\\
{Ivan Mati\'{c} \\
Department of Mathematics, University of Osijek \\ Trg Ljudevita
Gaja 6, Osijek, Croatia\\ E-mail: imatic@mathos.hr}
\end{flushleft}

\end{document}